\documentclass[a4paper]{article}
\usepackage{graphicx} 
\usepackage{placeins}

\AtEndDocument{%
  \par
  \medskip
  \begin{tabular}{@{}l@{}}%
    \textsc{Yau Mathematical Sciences Center, Tsinghua University, Beijing, China}, and\\
    \textsc{Department of Mathematics and Computer Science, Faculty of Science}\\ 
    \textsc{Chulalongkorn University, Bangkok, Thailand}\\
    \textit{E-mail address}: \texttt{supanat.k@chula.ac.th}
    \\\\
    \textsc{School of Mathematical Sciences, University of Science and Technology of China}\\
    \textsc{Hefei, China}\\
    \textit{E-mail address}: \texttt{spliu@ustc.edu.cn}
    \\\\
    \textsc{MPI MiS Leipzig, Leipzig, Germany}, and\\
    \textsc{Faculty of Mathematics and Computer Science, Leipzig University, Leipzig, Germany}\\
    \textit{E-mail address}: \texttt{cfmuench@gmail.com}
    \\\\
    \textsc{Department of Mathematical Sciences, Durham University, Durham, United Kingdom}\\
    \textit{E-mail address}: \texttt{norbert.peyerimhoff@durham.ac.uk}

  \end{tabular}}

\usepackage[backend=bibtex, 
style=numeric,
sortcites=true,
firstinits=true,
hyperref,
maxbibnames=99,
doi=false,isbn=false,url=false,eprint=false,
]{biblatex}

\bibliography{Bibliography}{}

\usepackage{amssymb,amsmath,amsfonts,amsthm}
\usepackage{latexsym}
\usepackage[shortlabels]{enumitem}
\usepackage{color}

\newtheorem{theorem}{Theorem}[section]

\newtheorem{lemma}[theorem]{Lemma}

\newtheorem{proposition}[theorem]{Proposition}

\newtheorem*{claim*}{Claim}

\newtheorem*{conjecture*}{Conjecture}

\theoremstyle{definition}
\newtheorem{definition}[theorem]{Definition}
\theoremstyle{remark}   
\newtheorem{remark}[theorem]{Remark}
\newtheorem*{remark*}{Remark}

\theoremstyle{remark}

\title{Entropic curvature not comparable to other curvatures -- or is it?}

\author{Supanat Kamtue \and Shiping Liu \and Florentin M\"unch \and Norbert Peyerimhoff}

\date{}

\begin{document}

\maketitle

\abstract{In this paper we consider global $\theta$-curvatures of finite Markov chains with associated means $\theta$ in the spirit of the entropic curvature (based on the logarithmic mean) by Erbar-Maas and Mielke. As in the case of Bakry-\'Emery curvature, we also allow for a finite dimension parameter by making use of an adapted $\Gamma$ calculus for $\theta$-curvatures. We prove explicit positive lower curvature bounds (both finite- and infinite-dimensional) for finite abelian Cayley graphs. In the case of cycles, we provide also an upper curvature bound which shows that our lower bounds are asymptotically sharp (up to a logarithmic factor). Moreover, we prove new universal lower curvature bounds for finite Markov chains as well as curvature perturbation results (allowing, in particular, to compare entropic and Bakry-\'Emery curvatures). Finally, we present examples where entropic curvature differs significantly from other curvature notions like Bakry-\'Emery curvature or Ollivier Ricci and sectional curvatures.}

\tableofcontents

\section{Introduction}

In \cite{EM-12}, Erbar and Maas introduced in 2012 entropic curvature for finite Markov chains $(X,Q,\pi)$ with transition probabilities $Q: X \times X \to [0,1]$ and steady state $\pi: X \to [0,1]$ by way of a particular mean $\theta_{log}$, called the logarithmic mean (see also \cite{Mi-13} for related work). In this paper, we will always assume that our finite Markov chains are irreducible and reversible, and the transition probabilities $Q$ induce a canonical weighted graph structure on $X$, where edges correspond to positive transition probabilities. Conversely, any weighted graph $G=(V,E)$ with vertex measure $m: V \to (0,\infty)$ and symmetric edge weights $w: V \times V \to [0,\infty)$ give rise to a canonical finite irreducible, reversible Markov chain with $Q(x,y) = \frac{w_{xy}}{m(x)}$ for $x \sim y$. For details about this, see Subsection \ref{subsec:markchandmeans}.
 
The notion of entropic curvature is \emph{non-local} in nature and generalizes naturally to other means $\theta: [0,\infty) \times [0,\times) \to [0,\infty)$, like the arithmetic or the geometric mean. In the case of the arithmetic mean, the resulting curvature is closely related to Bakry-\'Emery curvature. This relation was well-known and mentioned at various places, and was further manifested by a modified $\Gamma$ calculus (given in \cite{KLMP-23}) for entropic curvature and more general $\theta$-curvatures. The $\Gamma$ operators are defined locally and are crucial in the definition of the Bakry-\'Emery curvature, which is defined \emph{locally} on the vertices of a finite Markov chain and which includes an additional dimension parameter. This approach via modified $\Gamma$ operators motivates also the introduction of a dimension parameter into entropic curvature and into more general $\theta$-curvatures. The parameters $K \in \mathbb{R}$ and $N \in (0,\infty]$ in the $CD_\theta(K,N)$-property of a finite Markov chain $(X,Q,\pi)$ represent lower and upper bounds on global curvature and dimension, respectively. Since we will make also use of the $\Gamma$ formalism in this paper, we provide a brief overview about it in Subsection \ref{subsec:AdGamma}.

\subsection{Our results}

Our main concern are entropic and $\theta$-curvature bounds for finite Markov chains. In Section \ref{sec:Cayley} we consider finite Markov chains with an underlying abelian Cayley graph structure and derive positive lower curvature bounds for both finite and infinite dimensions as well as upper curvature bounds for cycles. A combined result derived from the lower curvature bounds is presented in the following theorem.

\begin{theorem}
  Let $(X,Q,\pi)$ be an irreducible, reversible finite Markov chain such that the induced structure on $X$ is an abelian Cayley graph ${\rm{Cay}}(\Gamma,S)$ with finitely presented group $\Gamma$ and finite symmetric generating set $S$, and with the additional property
  $$ Q(x,sx) = Q(s^{-1}x,x) > 0 \quad \text{for all $x \in X$ and all $s \in S$.} $$
  Let $Q_{\min} := \min_{s \in S} Q(x,sx)$. Let $r:=\max_{s \in S} {\rm{ord}}(s)$ be the maximal order of the generators, and let $S_2 \subset S$ be the subset of generators of order $2$. Then $(X,Q,\pi)$ satisfies $CD_{ent}(K,N)$ for the following pairs of $(K,N)$:
    \begin{align*}
        (K,N) &= \left( \frac{Q_{\min}}{50 r^4},\infty \right), \\
        (K,N) &= \left( 0, |S|-0.07|S\backslash S_2| \right), \quad \text{and} \\
        (K,N) &= \left( \frac{Q_{\min}}{100 r^4}, 2|S|-0.14|S\backslash S_2| \right).
    \end{align*}
\end{theorem}

In fact, the third pair $(K,N)$ in the theorem is deduced by the first two pairs by using the trade-off argument given in \cite[Remark 3.11]{KLMP-23} (with $\alpha=1/2$). The first two pairs are the results presented in Theorems \ref{thm:abellowbdinfdim} and \ref{thm:abellowbdfindim} in Subsections \ref{subsec:lowbdinf} and \ref{subsec:lowbdfin}, respectively. Note that $CD_{ent}(0,|S|-0.07|S\backslash S_2|)$ has a slightly better dimension parameter than the corresponding Bakry-\'Emery curvature-dimension condition $CD_a(0,|S|)$ for these abelian Cayley graph Markov chains (see Remark \ref{rem:comBEent}).

Non-negativity of entropic curvature for these abelian Cayley graph Markov chains was already proved by \cite[Proposition 5.4]{EM-12}. In fact, they formulated this result for Markov chains with certain \emph{mapping representations}, which is equivalent to the abelian Cayley graph structure (see Theorem \ref{thm:maprepabcay} in Subsection \ref{subsec:specmap}). The concept of a mapping representation, as introduced in \cite[Definition 5.2]{EM-12}, follows in spirit the conditions given in the earlier paper \cite{CDPP-09}. This earlier paper studies a Bochner-Bakry-\'Emery approach to the Modified Logarithmic Sobolev Inequality (MLSI) and considers particular examples: various birth-death processes and zero range processes, and Bernoulli-Laplace models which are equivalent to random walks on Johnson graphs. 

Moreover, abelian Cayley graph Markov chains have also non-negative Bakry-\'Emery and Ollivier Ricci curvatures (see, e.g., \cite{cushing2021curvatures} and references therein). The above theorem stating strictly positive entropic curvature for these Markov chains is surprising, since cycles of large enough diameter have always vanishing Bakry-\'Emery and Ollivier Ricci curvature. Moreover, we cannot deduce vanishing first homology from strictly positive entropic curvature (in contrast to Bakry-\'Emery and Forman curvature, see \cite{KMY-21,MR-20,For-03}). 

The following theorem, providing an upper curvature entropic curvature bound for cycles, shows that the lower curvature bound in the previous theorem is asymptotically sharp up to a factor $c (\log r )^2$. This result is a reformulation of Theorem \ref{thm:circleentbd} in Subsection \ref{subsec: uppbd-cycle}.

\begin{theorem} Let $G$ be a cycle graph of $4r$ vertices with $Q(x,y) = 1/2$ for neighbours $x \sim y$. Assume $r \ge 4$. This Markov chain does not satisfy $CD_{ent}(K,\infty)$ for any
$$ K > \frac{2500 (\log r)^2}{r^4}. $$
\end{theorem}

In Section \ref{sec:universal_bd_perturbed}, we consider general finite Markov chains and prove the following explicit lower entropic curvature bound (see Theorem \ref{thm:univgenlowbondfindim}; in the case of infinite dimension, we have a slightly better result, see Theorem \ref{thm:univgenlowbondinfdim}).

\begin{theorem}
  Let $(X,Q,\pi)$ be an irreducible, reversible finite Markov chain with 
  $$ 
  Q_{\min} := \min_{x \sim y} Q(x,y).
  $$
  Then $(X,Q,\pi)$ satisfies $CD_{ent}(-K,N)$ with
  $$ K = \frac{1}{2} + \frac{4}{Q_{\min}} \left( 1 + \frac{4}{N^2} \right) $$
  for all $N \in (0,\infty]$.
\end{theorem}

It was already shown in \cite[Theorem 4.1]{Mi-13} that the entropic curvature of every finite Markov chain is greater than $- \infty$. Our theorem provides an explicit lower bound and also holds under the stronger finite dimension assumption. The main ideas of our theorem were subsequently used to derive a lower bound for the intertwining curvature in \cite{MWZ-24}.  

The next part of the paper is concerned with perturbation results, comparing curvatures of pairs of Markov chains with associated means. The results are log-Lipschitz estimates. They allow to compare curvatures (like entropic curvature and Bakry-\'Emery curvature) associated to restricted probability densities with bounds on the logarithmic gradients (see Theorem \ref{thm:perturb1} in Subsection \ref{subsec:perturb}). The following version 
allows for different Markov chains with the same topology and the same mean (see Theorem \ref{thm:perturb2}). It is formulated in the terminology of weighted graphs.

\begin{theorem} Let $G=(V,E)$ be a connected finite graph with two choices of vertex measures $m, \widetilde m: B \to (0,1]$ and symmetric edge weights $w, \widetilde w: V \times V \to [0,1]$ satisfying $m(x) \ge \sum_y w(x,y)$ and $\widetilde m(x) \ge \sum_y \widetilde w(x,y)$. Let $(X,Q,\pi)$ and $(X,\widetilde Q,\widetilde \pi)$ be the Markov chains associated to $(V,E,m,w)$ and $(V,E,\widetilde m,\widetilde w)$, respectively.  
Assume there exists $\epsilon > 0$ such that
$$ \sup_{x \sim y} \left\vert \log \frac{\widetilde w_{xy}}{w_{xy}} \right \vert \le \epsilon \quad \text{and} \quad \sup_{x \in X} \left\vert \log \frac{\widetilde m(x)}{m(x)} \right\vert \le \epsilon. $$
Let $\theta$ be an arbitrary mean. Then we have 
$$ | K_\theta(X,Q,\pi)) - K_\theta(\widetilde X,\widetilde Q,\widetilde \pi) | \le 27 \epsilon \left( 1 + \frac{8}{Q_{\rm{min}}} \right), $$
where 
$$Q_{\rm{min}} := \inf_{x \sim y} \min\{ Q(x,y),\widetilde Q(x,y)\}.$$
\end{theorem}



\begin{table}[h]
    \centering
    \begin{tabular}{|c|c c c c|}
    \hline
    & ENT & BE & ORC & SEC\\
    \hline
    weighted 3-point graph & -- & + & + & -- 
    \\
    perturbed $C_6$ & + & -- & -- & -- 
    \\
    Layer-rescaled prism $C_5\times K_2$ & -- & -- & + & 0 \\
    \hline
    \end{tabular}
    \caption{Examples with various curvature signs. ``ENT'' stands for entropic curvature $K_{ent}$, ``BE'' stands for global Bakry-\'Emery curvature lower bound $K_{\theta_a}$ (See Definition \ref{defn: CDtheta} below) and ``ORC'' and ``SEC'' stand for Ollivier Ricci and Ollivier sectional curvature lower bound, respectively, which are introduced in Definition \ref{def:OllRicsec}.}
    \label{tab:examples}
\end{table}
\FloatBarrier

The penultimate section (Section \ref{sec:examples}) in this paper presents various examples for which entropic curvature has opposite sign to Bakry-\'Emery curvature as well as Ollivier curvatures (both Ricci and sectional; see Definition \ref{def:OllRicsec}). Curvature properties of these examples are presented in Table \ref{tab:examples}.

Finally, the paper ends with some interesting questions for further investigations in Section \ref{sec:questions}.

\subsection{Previous work on discrete curvature}
Discrete Ricci curvature has been a vibrant subject of research in the recent years. On manifolds, there is only one Ricci curvature notion, however on graphs there is multitude of analogues, namely Ollivier Ricci curvature, Bakry Emery Ricci curvature and entropic Ricci curvature.
They all have in common that a lower Ricci curvature bound $K$ can be characterized via semigroup estimates of the form
\[
\|\nabla P_t f\| \leq e^{-Kt} \|\nabla f\|
\]
for suitable norms and gradients.
The reason that the discrete Ricci curvature notions are different (and not even comparable), is the difference in measuring the gradient, and the lack of a chain rule for the gradient and Laplacian.
An overview on discrete Ricci curvature is given in the book \cite{najman2017modern}.
For work on Bakry Emery Ricci curvature, see e.g. 
\cite{fathi2015curvature,schmuckenschlager1998curvature,
lin2010ricci,KKRT-16}.
For work on Ollivier Ricci curvature, see e.g.
\cite{ollivier2007ricci,ollivier2009ricci,
lin2011ricci,fathi2022discrete}.
In this article, we are particularly interested in entropic Ricci curvature. For previous work, see 
\cite{FM-16,erbar2017ricci,erbar2018poincare,Erbar2019EntropicCA,Mi-13,KLMP-23}.
In contrast to the entropic curvature by Erbar, Maas and Mielke, there is also a different notion of entropic curvature introduced by Samson and based on Schrödinger bridges for the $\ell_1$-Wasserstein distance, see 
\cite{Sam-22,rapaport2023criteria,pedrotti2023contractive}.
Recently, Ricci curvature has been proven to be a useful tool in machine learning, and in particular, in graph neural networks 
\cite{topping2021understanding,nguyen2023revisiting,ye2019curvature}.

\section{Setup and Notation}


In this section, we briefly recall some fundamental notions used in connection with finite Markov chains $(X,Q,\pi)$ and entropic curvature and present an adapted $\Gamma$ calculus which will be used throughout this paper. 

\subsection{Markov chains and means}
\label{subsec:markchandmeans}

All Markov chains $(X,Q,\pi)$ in this paper will be \emph{finite}, \emph{irreducible}, and \emph{reversible}, and we will refer to them simply as \emph{finite Markov chains} by tacitly assuming these extra conditions. The \emph{Markov kernel} $Q: X \times X \to [0,1]$ satisfies $\sum_{y \in X} Q(x,y) = 1$ and the \emph{steady state} $\pi: X \to [0,1]$ can be viewed as a normalized left eigenvector of the transition probability matrix defined by $Q$, that is
$$ \sum_{x \in X} \pi(x) = 1, \quad \pi(y) = \sum_{x \in X} \pi(x) Q(x,y). $$
Reversibility is expressed by the condition
$$ Q(x,y)\pi(x) = Q(y,x)\pi(y). $$
Moreover, $Q$ induces a graph structure on $X$: $x,y \in X$ are adjacent (denoted by $x \sim y$) precisely if $Q(x,y) > 0$ and we have a symmetric weight function $w: X \times X \to [0,\infty)$ via $w(x,y) = Q(x,y)\pi(x)$, which is non-zero on the edges.

Conversely, any undirected connected finite weighted graph $X = (V,E)$ with vertex set $V$ and edge set $E$, and with \emph{vertex  measure} $m: V \to (0,\infty)$ and symmetric \emph{edge weights} $w: V \times V \to [0,\infty)$ satisfying $w(x,y) = w(y,x) > 0$ if and only if $x \sim y$ and $m(x) \ge \sum_{y \sim x} w(x,y)$ gives rise to a finite, irreducible, reversible Markov chain $(X,Q,\pi)$ by setting $Q(x,y) = w(x,y)/m(x)$ for $x \neq y$, $Q(x,x) = 1- \sum_{y \sim x} Q(x,y)$ for the transition probability, and $\pi(x) = m(x) / (\sum_{y \in V} m(y))$ for the steady state.

Natural operators associated to this Markov chain $(X,Q,\pi)$ are the (standard) \emph{Laplacian}
$$ \Delta f(x) = \sum_{y \in X} Q(x,y)(f(y)-f(x)) $$
for any function $f \in X^{\mathbb{R}}$, and the \emph{gradient}
$$ \nabla f(x,y) := \begin{cases} f(y)-f(x), & \text{if $x \sim y$}, \\ 0, & \text{if $x \not\sim y$}. \end{cases} $$
Note that we have $\nabla(y,x) = - \nabla(x,y)$. Functions $V: X \times X \to \mathbb{R}$ with $V(y,x) = - V(x,y)$ 
and $V(x,y) = 0$ if $x\not\sim y$ are called \emph{vector fields} on $X$.

For curvature notions, in particular entropic curvature, we need the notion of a mean, as introduced in \cite{Ma-11,EM-12}. The relevant mean for entropic curvature is called logarithmic mean.

\begin{definition}
    A function $\theta: [0,\infty) \times [0,\infty) \to [0,\infty)$ is called a \emph{mean}, if the following properties are satisfied:
    \begin{itemize}
      \item[(i)] $\theta$ is continuous and its restriction to $(0,\infty) \times (0,\infty)$ is smooth;
      \item[(ii)] $\theta$ is symmetric, that is, $\theta(r,s) = \theta(s,r)$;
      \item[(iii)] $\theta$ is monotone, that is, $\theta(r,s) \ge \theta(t,s)$ for $r \ge t$;
      \item[(iv)] $\theta$ is homogeneous, that is, $\theta(\lambda r,\lambda s) = \lambda \theta(r,s)$ for $\lambda > 0$,
      \item[(v)] $\theta$ is normalized, that is, $\theta(1,1)=1$.
    \end{itemize}
    The \emph{logarithmic mean} is defined as
    $$ \theta_{\log}(r,s) =\int_0^1 r^{1-p}s^p dp, $$
    and has the additional property
    \begin{itemize}
      \item[(vi)] $\theta(0,s) =0$ for all $s \ge 0$.
    \end{itemize}
\end{definition}

In the case $r,s >0$, $r \neq s$, the logarithmic mean can be written, alternatively, as
$$ \theta_{log}(r,s) = \frac{r-s}{\log(r)-\log(s)}, $$
which justifies its name. Another natural mean is the \emph{arithmetic mean} $\theta_a(r,s) = \frac{r+s}{2}$, which does not have property (vi) and for which we have
$$ \min(r,s) \le \theta_{log}(r,s) \le \theta_a(r,s) \le \max(r,s). $$
The arithmetic mean is relevant in connection to Bakry-\'Emery curvature. 

Furthermore, we define 
\[
\partial_1 \theta(r,s) := \frac{\partial}{\partial r}(r,s) \quad \text{and} \quad \partial_2\theta(r,s) := \frac{\partial}{\partial s}(r,s).
\]

\subsection{Relevant properties of the logarithmic mean}
Throughout this subsection, we let $\theta=\theta_{log}$. We introduce the following functions $b_0, b:\mathbb{R}_+^3\to \mathbb{R}$ as
    \begin{align} 
        b_0(\alpha,\beta,\gamma) &:= \partial_1\theta(\beta,\gamma)\alpha+\partial_2\theta(\beta,\gamma)\beta, \label{defn: b0-function} \\
        b(\alpha,\beta,\gamma) &:= b_0(\alpha,\beta,\gamma) - \theta(\alpha,\beta) \nonumber\\ &\phantom{:}= \partial_1\theta(\beta,\gamma)\alpha+\partial_2\theta(\beta,\gamma)\beta-\theta(\alpha,\beta) \ge 0. \label{eq: b-function-defn}
    \end{align}
The function $b$ is a crucial ingredient for our analysis later in Section \ref{sec:Cayley}. The function $b_0$ will be relevant in Section \ref{sec:universal_bd_perturbed}. In addition to its nonnegativity, which can be seen from \cite[Lemma 2.2]{EM-12}, we present useful properties and estimates for the function $b$ in the following proposition.
(Note that in the case of the arithmetic mean $\theta=\theta_a$, the corresponding function $b$, defined in \eqref{eq: b-function-defn} would vanish identically.) 

\begin{proposition} \label{prop: b-function} The function $b:\mathbb{R}_+^3\to \mathbb{R}$ satisfies the following properties.
    \begin{enumerate}[(a)]
        \item The function $b$ is homogeneous: $b(\lambda\alpha,\lambda\beta,\lambda\gamma)=\lambda b(\alpha,\beta,\gamma)$ for all $\lambda>0$.
        \item $b(\alpha,\beta,\gamma) \ge 0$.
        \item  $b(\alpha,\beta,\gamma)\ge  b(\gamma,\beta,\alpha)$ if $\alpha\ge \gamma$.
        \item We have
            \[
            b(\alpha,\beta,\gamma)+b(\gamma,\beta,\alpha)\ge 0.08\beta\left(\log\left(\frac{\alpha\gamma}{\beta^2}\right)\right)^2.
            \]
        \item For $\frac{\alpha}{\beta},\frac{\gamma}{\beta}\in [\frac{1}{256},256]$, we have
            \[
            b(\alpha,\beta,\gamma)+b(\gamma,\beta,\alpha)\le 34 \beta\left(\log\left(\frac{\alpha\gamma}{\beta^2}\right)\right)^2.
            \]
        \item We have 
            \[\theta(\alpha,\beta)+\theta(\beta,\gamma)+\dfrac{b(\alpha,\beta,\gamma)\cdot b(\gamma,\beta,\alpha)}{b(\alpha,\beta,\gamma)+  b(\gamma,\beta,\alpha)} \ge \frac{1}{C}\beta,
            \] for some constant $C\le 0.93$. In particular, if $\alpha=\gamma$, we have \[2\theta(\alpha,\beta)+\frac{1}{2}b(\alpha,\beta,\alpha) \ge \frac{1}{C}\beta.\]
            \item $b(\alpha,\beta,\alpha)=b(\beta,\alpha,\beta)$, and
            \[ 4\theta(\alpha,\beta)+b(\alpha,\beta,\alpha)\ge 2(\alpha+\beta)\,.
            \]
    \end{enumerate}
\end{proposition}
\begin{proof}
Statement (a) is straightforward to check. The other statements can be proved without loss of generality under the assumption that $\beta=1$.
    
Recall $\theta(u,v)=\int_0^1 u^pv^{1-p}dp$.
Then
\begin{align} \label{eq:b-alpha-gamma}
b(\alpha,1,\gamma)
&=\partial_1\theta(1,\gamma)\alpha+\partial_2\theta(1,\gamma)-\theta(\alpha,1) \nonumber\\
&=\int_0^1 \Big( p\gamma^{1-p}\alpha+(1-p)\gamma^{-p}-\alpha^p \Big)dp \nonumber\\
&=\int_0^1 \alpha^p \Big( p(\alpha\gamma)^{1-p}+(1-p)(\alpha\gamma)^{-p}-1 \Big)dp,
\end{align} where the term in the large bracket is symmetric in $\alpha$ and $\gamma$, and it is nonnegative due to Jensen's inequality. Thus (b) and (c) follow immediately. Furthermore, the fact that $\alpha^p+\gamma^p\ge 2(\sqrt{\alpha\gamma})^p$ implies 
\[
b(\alpha,1,\gamma) + b(\gamma,1,\alpha) \ge 2b(\sqrt{\alpha\gamma},1,\sqrt{\alpha\gamma}).
\] To minimize the above term on the right hand side, we let $\alpha\gamma=e^s$, $s\in \mathbb{R}$ and define
\[
f(s):=\frac{2b(e^{s/2},1,e^{s/2})}{s^2}=
16s^{-4}e^{s/4}\sinh\left(\frac{s}{4}\right)\left(2\sinh\left(\frac{s}{2}\right)-s\right)\,.
\] 
WolframAlpha suggests that $f$ is convex and
$\min f=f(-5.8495)=0.0823$. Thus (d) is proved.

For statement (e), we consider the function
$$ g(s,t) := \frac{b(e^s,1,e^t)+b(e^t,1,e^s)}{(s+t)^2} $$
over the domain $(s,t) \in [-8\log(2),8\log(2)]\times[-8\log(2),8\log(2)]$. Figure \ref{fig:prop2.2e} shows that the maximum of $g(s,t)$ assumed at $(s,t)=(8\log(2),8\log(2))$ with $8\log(2) \approx 5.545$. Moreover, we have
$$ g(s,s) = \frac{e^{2s}-(2s+1)e^s+2s-1+e^{-s}}{2s^4} $$
and $g(8\log(2),8\log(2))$ is approximately $33.026$. This finishes (e).
\begin{figure}[h!]
\begin{center}
\includegraphics[width=0.3\textwidth]{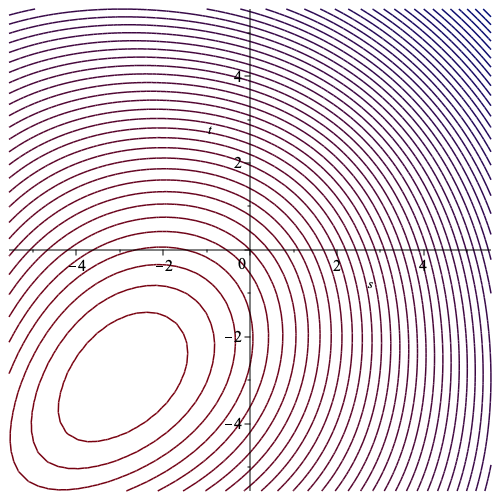}
\includegraphics[width=0.38\textwidth]{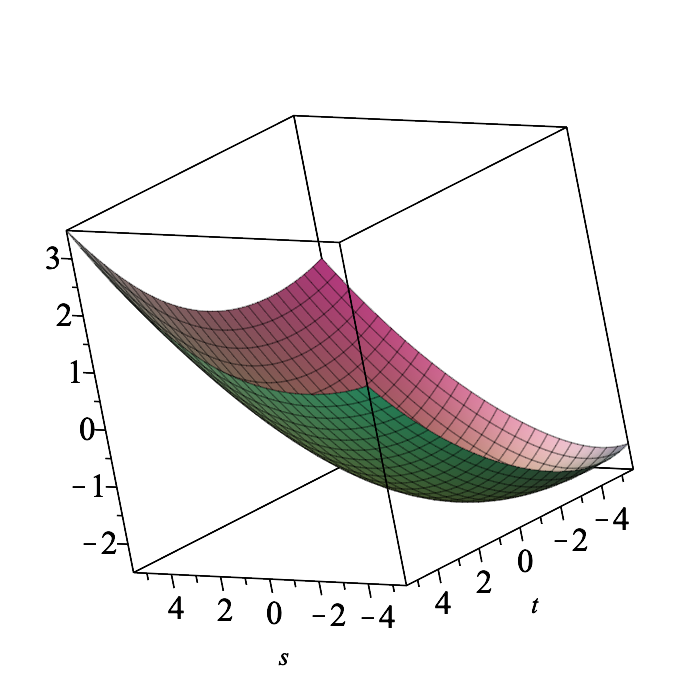}
\includegraphics[width=0.3\textwidth]{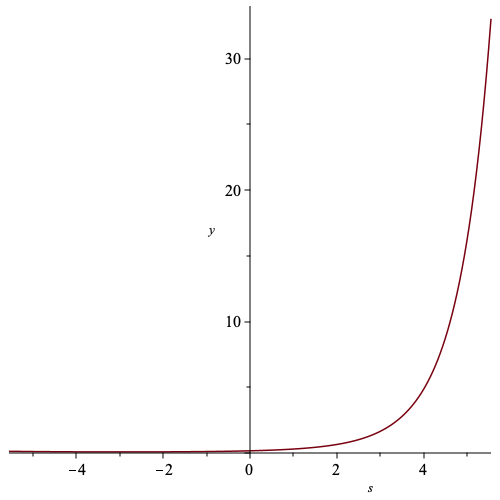}
\end{center}
\caption{Contour plot and plot of the function $\log(g(s,t))$ over the domain $[-8\log(2),8\log(2)]\times[-8\log(2),8\log(2)]$, respectively, on the left. The values of the contours increase the closer they are to the point $(s,t)=(8\log(2),8\log(2))$.
The graph of $g(s,s)$ over $[-8\log(2),8\log(2)]$ is plotted on the right.}
\label{fig:prop2.2e}
\end{figure}
\FloatBarrier

Now for statement (f), we need to prove that
\[
\theta(\alpha,1)+\theta(\gamma,1)+\frac{b(\alpha,1,\gamma)\cdot b(\gamma,1,\alpha)}{b(\alpha,1,\gamma) + b(\gamma,1,\alpha)}\ge 1.08.
\]
Denote $S:=b(\alpha,1,\gamma)$ and $T:=b(\gamma,1,\alpha)$. Assume without loss of generality that $\alpha\ge \gamma$. We learn from $(c)$ that $S\ge T$. We consider two separated cases: $S/T \ge 2$ and $S/T\le 2$.

\medskip

{\bf Case $S\ge 2T$}: we have $\frac{ST}{S+T} \ge \frac{2}{3}T$. Then
\begin{align*}
\theta(\alpha,1)+\theta(\gamma,1)+\frac{ST}{S+T} 
&\ge \theta(\alpha,1)+\theta(\gamma,1)+\frac{2}{3}b(\gamma,1,\alpha) \\
&= \theta(\alpha,1)+\frac{1}{3}\theta(\gamma,1)+\frac{2}{3}\partial_1\theta(1,\alpha)\gamma+\frac{2}{3}\partial_2\theta(1,\alpha)\\
&\ge \theta(\alpha,1)+\frac{2}{3}\partial_2\theta(1,\alpha).
\end{align*}

\medskip

{\bf Case $S \le 2T$}: we have $\frac{ST}{S+T} \ge \frac{1}{3}S$. Then
\begin{align*}
\theta(\alpha,1)+\theta(\gamma,1)+\frac{ST}{S+T} 
&\ge \theta(\alpha,1)+\theta(\gamma,1)+\frac{1}{3}b(\alpha,1,\gamma) \\
&= \frac{2}{3}\theta(\alpha,1)+\theta(\gamma,1)+\frac{1}{3}\partial_1\theta(1,\gamma)\alpha+\frac{1}{3}\partial_2\theta(1,\gamma)\\
&\ge \frac{5}{3}\theta(\gamma,1)+\frac{1}{3}\partial_1\theta(1,\gamma)\gamma+\frac{1}{3}\partial_2\theta(1,\gamma).
\end{align*}

Define 
\begin{align*}
    g(s)&:=\theta(e^s,1)+\frac{2}{3}\partial_2\theta(1,e^s)=\frac{e^s-1}{s}+\frac{2}{3}\frac{e^{-s}+s-1}{s^2}\\
    h(s)&:=\frac{5}{3}\theta(e^s,1)+\frac{1}{3}\partial_1\theta(1,e^s)e^s+\frac{1}{3}\partial_2\theta(1,e^s)\\&\phantom{;}=\frac{5}{3}\frac{e^s-1}{s}+\frac{1}{3}\frac{e^{s}-s-1}{s^2}e^s+\frac{1}{3}\frac{e^{-s}+s-1}{s^2}
\end{align*}

WolframAlpha suggests that $g$ and $h$ are both convex, and 
$$\min g=g(-1.29549)=1.09974 \quad \text{and} \min h=h(-2.38889)=1.08041.$$ 
We can then conclude from both cases that 
\begin{equation*}
\theta(\alpha,1)+\theta(\gamma,1)+\frac{b(\alpha,1,\gamma)\cdot b(\gamma,1,\alpha)}{b(\alpha,1,\gamma) + b(\gamma,1,\alpha)}\ge 1.08.
\end{equation*}
The inequality (g) follows from 
\begin{multline*} 
4 \theta(e^s,1)+b(e^s,1,e^s) - 2(e^s+1) \\ = \frac{e^{2s}-(2s^2-2s+1)e^s-(2s^2+2s+1)+e^{-s}}{s^2} \ge 0. 
\end{multline*}
\end{proof}

The following inequalities regarding the function $b_0$ are important in Section \ref{sec:universal_bd_perturbed}.

\begin{lemma} \label{lem:ineqforgenlowbd}
  We have the following inequalities for $\alpha, \beta, \gamma > 0$:
  \begin{align}
      b_0(\gamma,\beta,\alpha) \theta(\gamma,\beta) + b_0(\alpha,\beta,\gamma) \theta(\alpha,\beta) &\ge (\theta(\alpha,\beta))^2 + (\theta(\gamma,\beta))^2, \label{eq:b0est1} \\
      b_0(\gamma,\beta,\alpha) \theta(\gamma,\beta) + b_0(\alpha,\beta,\gamma) \theta(\alpha,\beta) &\ge \frac{\beta^2}{2}.
      \label{eq:b0est2}
  \end{align}
\end{lemma}

\begin{proof}
  Inequality \eqref{eq:b0est1} follows directly from the non-negativity of $b$.

For the proof of \eqref{eq:b0est2},
we can assume $\alpha = e^s$, $\beta=1$, and $\gamma = e^t$, by homogeneity of the inequality. Then the inequality translates into
\begin{multline*}
\left(\left(-\frac{1}{s}+\frac{e^s-1}{s^2}\right)e^t+\left(\frac{1}{s}-\frac{1-e^{-s}}{s^2}\right)\right)\frac{e^t-1}{t} + \\ \left(\left(-\frac{1}{t}+\frac{e^t-1}{t^2}\right)e^s+\left(\frac{1}{t}-\frac{1-e^{-t}}{t^2}\right)\right)\frac{e^s-1}{s} \ge \frac{1}{2}.
\end{multline*}
Introducing the function 
$$ k(t):=-\frac{1}{t} + \frac{e^t-1}{t^2}, $$
the inequality can be reformulated as follows:
\begin{equation} \label{eq:kineq} 
(k(s)e^t+k(-s))(tk(t)+1) + (k(t)e^s+k(-t))(sk(s)+1) \ge \frac{1}{2}. 
\end{equation}
It is easy to see that $k(t) \ge 0$ for all $t \in \mathbb{R}$, $k(0) = 1/2$, that both $k$ and $tk(t)+1=\frac{e^t-1}{t} $ are monotone increasing, and that
$$ t k(t) + 1 \ge 0 \quad \text{for all $t \in \mathbb{R}$}. $$ 

\begin{itemize}
\item Assume first that precisely one of $s,t$ is non-positive. Let us assume, without loss of generality, that we have $s \le 0$ and $t > 0$. Then we have
\begin{multline*}
(k(s)e^t+k(-s))(tk(t)+1) + (k(t)e^s+k(-t))(sk(s)+1) \\ \ge k(-s)(tk(t)+1) \ge k(-s) \ge k(0) = 1/2. 
\end{multline*}
This shows the inequality as soon as precisely one of $s,t$ is non-positive. 
\item Assume now that both $s,t$ are positive. Then we have
\begin{multline*}
(k(s)e^t+k(-s))(tk(t)+1) + (k(t)e^s+k(-t))(sk(s)+1) \\ \ge k(s)e^t(tk(t)+1) \ge k(s) \ge k(0) = 1/2. 
\end{multline*}
\item Finally, assume that $s,t \le 0$. Without loss of generality, we can assume $s \le t \le 0$. Then we have, by monotonicity of $tk(t)+1$ and $k(-t) \ge 1/2$,
\begin{multline*}
(k(s)e^t+k(-s))(tk(t)+1) + (k(t)e^s+k(-t))(sk(s)+1) \\ \ge k(-s)(tk(t)+1) + k(-t)(sk(s)+1) \\
\ge k(-s)(sk(s)+1) + \frac{1}{2}(sk(s)+1) = \left(\frac{1}{2} + k(-s)\right)(sk(s)+1) \\
\left( \frac{1}{2} + \frac{1}{s} + \frac{e^{-s}-1}{s^2} \right) \frac{e^s-1}{s}.
\end{multline*}
By replacing $s$ by $-s$, it suffices to show for $s\ge 0$ that
$$ g(s) := \frac{1-e^{-s}}{s} \left( \frac{e^s-1}{s^2} + \frac{1}{2} - \frac{1}{s} \right) \ge \frac{1}{2}. $$
In the case $s \ge 2$, we have, by Taylor expansion
$$ g(s) \ge \frac{(1-e^{-s})(e^s-1)}{s^3} = 2 \frac{\cosh(s)-1}{s^3} \ge \frac{2}{s^3}\left( \frac{s^2}{2} + \frac{s^4}{24}   \right) = \frac{1}{s} + \frac{s}{12} \ge \frac{1}{2}. $$
We have also, by Taylor expansion, for $s \ge 0$
$$ g(s) \ge \frac{1-e^{-s}}{s^3} \left( s^2 + \frac{s^3}{6} \right), $$
and it suffices to show for $0 \le s \le 2$ that
$$ \frac{1-e^{-s}}{s^3} \left( s^2 + \frac{s^3}{6} \right) \ge \frac{1}{2}. $$
This is equivalent to
$$ e^s\left( 1- \frac{s}{3} \right) \ge 1 + \frac{s}{6}, $$
which holds if we have
$$ \left( 1 + s + \frac{s^2}{2} \right) \left( 1- \frac{s}{3} \right) \ge 1 + \frac{s}{6} $$
for $0 \le s \le 2$. This, in turn, is equivalent to
$$ \frac{1}{2} \ge \frac{s(s-1)}{6} $$
for $0 \le s \le 2$, which is obviously true.
\end{itemize}
These case investigations show that inequality \eqref{eq:kineq} holds for all $s,t \in \mathbb{R}$, which finishes the proof of inequality \eqref{eq:b0est2}.
\end{proof}

\begin{remark}
The following illustration indicates that   
\eqref{eq:b0est2} holds even with a better lower bound $C$ slightly below $1$ instead of $1/2$. Since both sides of \eqref{eq:b0est2} are homogeneous of degree $2$, it suffices for its proof to show, for all $s,t \in \mathbb{R}$,
  $$ b_0(e^t,1,e^s) \theta(e^t,1) + b_0(e^s,1,e^t) \theta(e^s,1) \ge C. $$
  Figure \ref{fig:pic-lemma4.4} illustrates that this latter inequality.
\begin{figure}[h!]
\begin{center}
\includegraphics[width=0.38\textwidth]{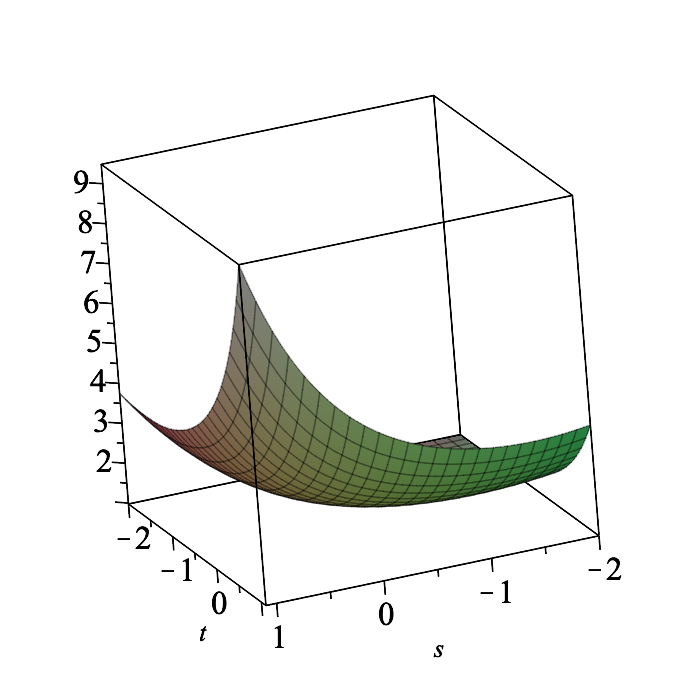}
\includegraphics[width=0.3\textwidth]{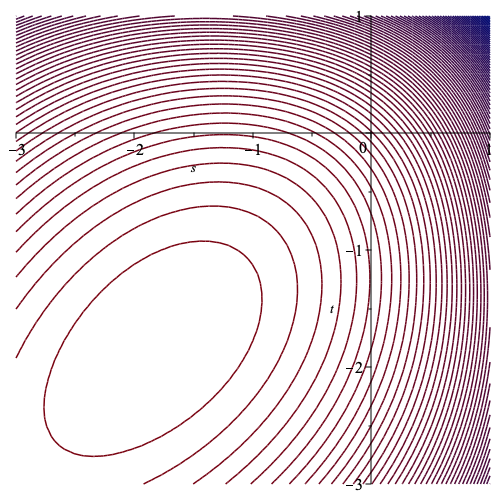}
\includegraphics[width=0.3\textwidth]{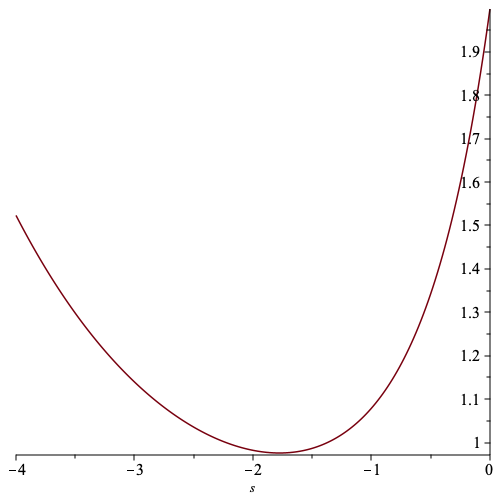}
\end{center}
\caption{Plot and contour plot of $b_0(e^t,1,e^s)\theta(e^t,1)+b_0(e^s,1,e^t)\theta(e^s,1)$ on the domains $[-2,1.1]\times[-2,1.1]$ and $[-3,1]\times[-3,1]$, respectively, on the left, and graph of the restriction of this function to the diagonal $s=t$ on $[-4,0]$ on the right.}
\label{fig:pic-lemma4.4}
\end{figure}
\FloatBarrier
\end{remark}

\subsection{Entropic curvature}

In analogy to the Ricci curvature notion by Sturm \cite{St-06-i,St-06-ii} and Lott-Villani \cite{LV-09} for metric measure spaces, entropic curvature is defined via a convexity property along geodesics with respect to suitably modified Wasserstein metric, adapted to the discrete setting of a finite Markov chain. For more details, we refer readers to the original papers \cite{Ma-11,EM-12}. In this paper, we focus on an equivalent infinitesimal curvature formulation of entropic curvature (see Definition \ref{def:entcurvclassical} below), which is an integrated Bochner-type inequality and involves two operators, denoted by $\mathcal{A}_\rho$ and $\mathcal{B}_\rho$ for all $\rho \in \mathcal{P}_*(X)$, where
$$ \mathcal{P}(X) = \{ \rho \in [0,\infty)^X: \sum_{x \in X} \rho(x) \pi(x) = 1 \}, $$
and $\mathcal{P}_*(X)$ is the subset of strictly positive probability densities $\rho \in (0,\infty)^X$. These operators are defined as follows for functions $f,g \in \mathbb{R}^X$ and arbitrary but fixed mean $\theta$:
\begin{align}
    (\widehat \Delta \rho)(x,y) &:= \partial_1 \theta(\rho(x),\rho(y))(\Delta \rho)(x) + \partial_2 \theta(\rho(x),\rho(y))(\Delta \rho)(y), \nonumber \\  
    \mathcal{A}_\rho(f,g) &:= \langle \nabla f, \nabla g \rangle_\rho, \label{eq:Arhofg}\\
    \mathcal{A}_\rho(f) &:= \mathcal{A}_\rho(f,f) = \Vert \nabla f \Vert_\rho^2, \nonumber \\
    \mathcal{B}_\rho(f) &:= \frac{1}{2} \langle \widehat \Delta \rho \cdot \nabla f, \nabla f \rangle_\pi - \langle \hat \rho \cdot \nabla f, \nabla(\Delta f) \rangle_\pi \nonumber
\end{align}
with 
$$ \partial_1 \theta(r,s) := \frac{\partial}{\partial r}(\theta(r,s) \quad \text{and} \quad \partial_2\theta(r,s) := \frac{\partial}{\partial s}\theta(r,s), $$
and where we used the following inner products
\begin{align*}
   \langle V_1,V_2 \rangle_\pi &:= \frac{1}{2} \sum_{x,y \in X} V_1(x,y)V_2(x,y)Q(x,y)\pi(x), \\
   \langle V_1, V_2 \rangle_\rho &:= \langle \hat \rho V_1, V_2 \rangle_\pi = \frac{1}{2} \sum_{x,y \in X} \widehat \rho(x,y) V_1(x,y) V_2(x,y) Q(x,y) \pi(x) 
\end{align*}
for vector fields $V_1, V_2: X \times X \to \mathbb{R}$ and
$$ \widehat \rho(x,y) = \widehat \rho_{xy} = \theta(\rho(x),\rho(y)). $$
Moreover, we have for any mean (see \cite[Lemma 2.2]{EM-12}):
\begin{equation} \label{eq:rhoasderiv}
   \widehat \rho_{xy} := \widehat \rho(x,y) = \theta(\rho(x),\rho(y)) = \partial_1 \theta(\rho(x),\rho(y)) \rho(x) + \partial_2 \theta(\rho(x),\rho(y)) \rho(y). 
\end{equation}

We use the following infinitesimal definition of entropic curvature (combining \cite[Definition 1.1]{EM-12} and \cite[Theorem 4.5]{EM-12}):

\begin{definition} \label{def:entcurvclassical}
Let $(X,Q,\pi)$ be an irreducible, reversible finite Markov chain and $K \in \mathbb{R}$. We say that $(X,Q,\pi)$ has (non-local) \emph{entropic curvature bounded below by $K$} if we have, for all $\rho \in \mathcal{P}_*(X)$ and all $f \in {\mathbb{R}}^X$,
$$ \mathcal{B}_\rho(f) \ge K \mathcal{A}_\rho(f), $$
where the operators $\mathcal{A}_\rho, \mathcal{B}_\rho$ are defined with respect to the logarithmic mean $\theta = \theta_{log}$. Moreover, we define the \emph{entropic curvature} of $(X,Q,\pi)$ as follows:
$$ K_{ent}(X,Q,\pi) = \inf_{\substack{\rho, f\\ \mathcal{A}_\rho(f) \neq 0}} \frac{\mathcal{B}_\rho(f)}{\mathcal{A}_\rho(f)}, $$
where $\rho$ is chosen in $\mathcal{P}_*(X)$.
\end{definition}

Of course, analogous non-local curvature notions can be also introduced for any other mean $\theta$. We will refer to these curvature notions henceforth as $\theta$-curvature.

\subsection{Adapted $\Gamma$ calculus and $\theta$-curvatures}
\label{subsec:AdGamma}

Let $(X,Q,\pi)$ be a finite Markov chain and $\theta: [0,\infty) \times [0,\infty) \to [0,\infty)$ be a mean.
In analogy to the Bakry-\'Emery $\Gamma$ calculus (see \cite{BE-85}), we introduce for every $\rho \in (0,\infty)^X$ the following $\rho$-Laplacian and modified bilinear operators $\Gamma_\rho, \Gamma_{2,\rho}$ (see \cite[Definition 3.1]{KLMP-23}):
\begin{align}
    \Delta_\rho f(x) &:= \sum_{y \in X} 2 \partial_1\theta(\rho_x,\rho_y)(f(y)-f(x))Q(x,y), \nonumber \\
    \Gamma_\rho(f_1,f_2)(x) &:= \frac{1}{2} \left( \Delta_\rho(f_1f_2)(x) - (f_1\Delta_\rho f_2)(x) - (f_2\Delta_\rho f_1)(x) \right) \label{eq:Gammarho} \\
    &= \sum_{y \in X} \partial_1\theta(\rho_x,\rho_y)(f_1(y)-f_1(x))(f_2(y)-f_2(x))Q(x,y), \nonumber \\
    \Gamma_{2,\rho}(f_1,f_2)(x)  &:= \frac{1}{2} \left( \Delta \Gamma_\rho(f_1,f_2)(x) - \Gamma_\rho(f_1,\Delta f_2)(x)-\Gamma_\rho(f_2,\Delta f_1)(x) \right) \label{eq:Gammarho2} 
\end{align}
for all $f,f_1,f_2 \in \mathbb{R}^X$. Moreover, we set 
$$ \Gamma_\rho f := \Gamma_\rho(f,f) \quad \text{and} \quad \Gamma_{2,\rho}(f) := \Gamma_{2,\rho}(f,f). $$

Note that the definition of $\Gamma_\rho$ involves the $\rho$-Laplacian, while $\Gamma_{2,\rho}$ involves the standard Laplacian. We like to emphasize that the operators $\Gamma_\rho$ and $\Gamma_{2,\rho}$ reduce to the classical $\Gamma$ and $\Gamma_2$ operators, involving only the standard Laplacian, if we choose $\theta = \theta_a$, that is, the arithmetic mean, since in this case we have $2 \partial_1\theta(r,s) = 1$ for all $r,s \ge 0$. It was shown in \cite[Section 3.1]{KLMP-23} that we have
\begin{align}
  \mathcal{A}_\rho(f) &= \langle \rho, \Gamma_\rho f \rangle_\pi, \label{eq:ArhoGamma}\\
  \mathcal{B}_\rho(f) &= \langle \rho, \Gamma_{2,\rho} f \rangle_\pi, \nonumber
\end{align}
with the following inner product 
$$ \langle f_1, f_2 \rangle_\pi := \sum_{x \in X} f_1(x) f_2(x) \pi(x) $$
for functions $f_1,f_2 \in \mathbb{R}^X$. By polarization, we obtain
$$ \mathcal{A}_\rho(f_1,f_2) = \langle \rho, \Gamma_\rho(f_1,f_2) \rangle_\pi $$
and 
\begin{equation} \label{eq:BrhoArhoLaprel} 
\mathcal{B}_\rho(f) = \langle \rho, \Gamma_{2,\rho} f \rangle_\pi =\frac{1}{2} \langle \Delta \rho, \Gamma_\rho f \rangle_\pi - \mathcal{A}_{\rho}(f,\Delta f), 
\end{equation}
using the fact that $\Delta$ is a symmetric operator with respect to $\langle \cdot,\cdot \rangle_\pi$.
Classical Bakry-\'Emery curvature at a vertex $x \in X$ is defined by the following local Bochner-type curvature-dimension inequality.

\begin{definition}[see {\cite[Definition 2.4]{KLMP-23}}] \label{def:BEcurv}
  Let $(X,Q,\pi)$ be n irreducible, reversible finite Markov chain and $N \in (0,\infty]$ be a dimension parameter. The \emph{Bakry-\'Emery curvature} $K_N(x)$ at a vertex $x \in X$ is the supremum of all $K \in \mathbb{R}$ satisfying
  $$ \Gamma_2 f(x) \ge \frac{1}{N} (\Delta f(x))^2 + K \Gamma f(x) \quad \text{for all $f \in \mathbb{R}^X$,}  $$
  where $\Gamma$ and $\Gamma_2$ are the bilinear operators in \eqref{eq:Gammarho} and \eqref{eq:Gammarho2}, respectively, in the special case $\theta = \theta_a$. We also say that $(X,Q,\pi)$ \emph{satisfies the (global) Bakry-\'Emery curvature-dimension inequality $CD_a(K,N)$} (where the index $a$ indicates the involvement of the mean $\theta = \theta_a$), if we have $K_N(x) \ge K$ for all $x \in X$.
\end{definition}

In view of this, it makes sense to describe nonlocal $\theta$-curvature for an arbitrary mean $\theta$ with an additional dimension parameter $n$ as follows.

\begin{definition} \label{defn: CDtheta}
    Let $(X,Q,\pi)$ be an irreducible, reversible finite Markov chain and $\theta: [0,\infty) \times [0,\infty) \to [0,\infty)$ a mean. We say that $(X,Q,\pi)$ has (non-local) \emph{$\theta$-curvature bounded below by $K \in \mathbb{R}$} for dimension $N \in (0,\infty]$ if we have, for all $\rho \in \mathcal{P}_*(X)$ and all $f \in \mathbb{R}^X$,
    \begin{equation} \label{eq:intCDcond} 
    \langle \rho, \Gamma_{2,\rho} f \rangle_\pi \ge \frac{1}{N} \langle \rho, (\Delta f)^2 \rangle_\pi + K \langle \rho, \Gamma_\rho f \rangle_\pi. 
    \end{equation}
    In this case, we also say that $(X,Q,\pi)$ \emph{satisfies $CD_\theta(K,N)$}. In the case $\theta = \theta_{log}$, we also write simply $CD_{ent}(K,N)$ for $CD_{\theta_{log}}(K,N)$. Moreover, we define the \emph{$\theta$-curvature} of $(X,Q,\pi)$ (for dimension $N = \infty$) as follows:
    \begin{equation} \label{eq:Ktheta}
    K_\theta(X,Q,\pi) = 
    \inf_{\substack{\rho, f\\ \langle \rho, \Gamma_\rho f \rangle_\pi \neq 0}} \frac{\langle \rho,\Gamma_{2,\rho} f \rangle_\pi}{\langle \rho,\Gamma_\rho f \rangle_\pi} = \inf_{\substack{\rho, f\\ \mathcal{A}_\rho(f) \neq 0}} \frac{\mathcal{B}_\rho(f)}{\mathcal{A}_\rho(f)},
    \end{equation}
    where $\rho$ is chosen in $\mathcal{P}_*(X)$. Similarly, we use the simplified notation $K_{ent}(X,Q,\pi)$ for $K_{\theta_{log}}(X,Q,\pi)$.
\end{definition}

We can think of \eqref{eq:intCDcond} as a (non-local) integrated curvature-dimension inequality against the probability density $\rho$, and we recover Definition \ref{def:entcurvclassical} via the special choices $\theta = \theta_{log}$ and $N = \infty$.
Moreover, as stated in \cite[Lemma 4.1]{KLMP-23}, the best lower $\theta_a$-curvature bound for $(X,Q,\pi)$ is given by $\min_{x \in X} K_N(x)$, where $K_N(x)$ is the (local) Bakry-\'Emery curvature at $x \in X$ introduced in Definition \ref{def:BEcurv}.  

\section{Positive curvature bounds for abelian Cayley graphs} \label{sec:Cayley}

In this section, we prove that every finite Markov chain $(X,Q,\pi)$ with underlying abelian Cayley graph structure with symmetric generating set $S$ and subset $S_2 \subset S$ of generators of order $2$ has strictly positive entropic curvature: it satisfies $CD_{ent}(K,\infty)$ with
\[ K=\frac{Q_{\min}}{50r^4}, \] where $r$ is the maximal order of its generators (see Theorem \ref{thm:abellowbdinfdim}). Moreover, we prove that it has non-negative finite-dimensional entropic curvature: it satisfies $CD_{ent}(0,N)$ with $N=|S|-0.07(|S\backslash S_2|)\le |S|$ (see Theorem \ref{thm:abellowbdfindim}). This should be compared with the Bakry-\'Emery curvature property $CD_a(0,|S|)$ (see Remark \ref{rem:comBEent}).  In Subsection \ref{subsec: uppbd-cycle}, we also derive an upper bound for entropic curvature of the cycle graph $C_r$ with the simple random walk and large enough $r$ (see Theorem \ref{thm:circleentbd}) and conclude that 
\[ r^{-4} \lesssim K_{ent}(C_r) \lesssim r^{-4}(\log r)^2.\]

\subsection{Special mapping representations and abelian Cayley graphs}
\label{subsec:specmap}

Let us recall the following definition of a \emph{mapping representation} of an irreducible, reversible finite Markov chain.

\begin{definition} [see {\cite[Definition 5.2]{EM-12}}]\label{def:maprep}
Let $(X,Q,\pi)$ be an irreducible, reversible, finite Markov chain.
A \emph{mapping representation} of $(X,Q,\pi)$ consists of a set $G$ of maps $\delta: X \to X$ and a function $c: X \times G \to [0,\infty)$ with the following properties:
\begin{itemize}
  \item[(a)] The Laplacian can be written as
  $$ \Delta f(x) = \sum_{\delta \in G} c(x,\delta) \nabla_\delta f(x),  $$
  where 
  $$ \nabla_\delta f(x) = f(\delta x) - f(x); $$
  \item[(b)] For every $\delta \in G$ there exists a \emph{unique} inverse of $\delta$, which we denote by $\delta^{-1} \in G$. Here we call $\eta \in G$ an  \emph{inverse} of $\delta$ if we have $\eta(\delta x) = x$ for all $x \in X$ with $c(x,\delta) > 0$;
  \item[(c)] For every $F: X \times G \to \mathbb{R}$,
  $$ \sum_{x \in X, \delta \in G} F(x,\delta)c(x,\delta) \pi(x) = \sum_{x \in X, \delta \in G} F(\delta x,\delta^{-1})c(x,\delta)\pi(x). $$
\end{itemize}
\end{definition}
Erbar and Maas derive the following explicit (non-negative) lower curvature bound for irreducible, reversible finite Markov chains admitting particular mapping representations.
\begin{theorem}[see {\cite[Proposition 5.4]{EM-12}}] \label{thm:EMmaprep} Let $(X,Q,\pi)$ be an irreducible, reversible finite Markov chain admitting a mapping representation $(G,c)$ with the following properties:
\begin{itemize}
\item[(i)] $\delta \circ \eta = \eta \circ \delta$ for all $\delta,\eta \in G$,
\item[(ii)] $c(\delta x,\eta) = c(x,\eta)$ for all $x \in X$ and $\delta,\eta \in G$.
\end{itemize}
Then we have
$$ {\rm{Ric}}(X,Q,\pi) \ge 0. $$
If $(G,c)$ satisfies additionally
\begin{itemize}
 \item[(iii)] $\delta \circ \delta = {\rm{id}}$ for all $\delta \in G$,   
\end{itemize}
then we have
$$ {\rm{Ric}}(X,Q,\pi) \ge 2 C $$
with
$$ C := \min\{ c(x,\delta) \mid \text{$x \in X$ and $\delta \in G$ 
 with $c(x,\delta) >0$} \}. $$
\end{theorem}

An example of a mapping representation $(G,c)$ satisfying (i) and (ii) is given in the case when the underlying combinatorial graph is a $d$-regular abelian Cayley graph ${\rm{Cay}}(X,S)$ with a symmetric set of generators $S = \{s_1,\dots,s_d\}$ (none of the $s_i$ is the identity) by setting $G := \{ \delta_1,\dots,\delta_d\}$, $\delta_i x = x s_i$ and $c(x,\delta_i) = Q(x,\delta_i x)$. With this choice, $(G,c)$ satisfies (iii) if and only if all generators $s_i$ are of order $2$, which is a very special case. Examples for that are Cayley graphs of Coxeter groups with the standard set of generators.

In this subsection, we show that the underlying combinatorial graph of any mapping representation $(G,c)$ satisfying (i) and (ii) has to be an abelian Cayley graph.  

\begin{theorem} \label{thm:maprepabcay}
Let $(G,c)$ be a mapping representation of an irreducible, reversible finite Markov chain $(X,Q,\pi)$ with condition $(i)$ and $(ii)$ from
Theorem \ref{thm:EMmaprep}.
  Then the underlying combinatorial graph $X$ carries the structure of an abelian Cayley graph with $c(x,\delta)=c(x,\delta^{-1})$ for all $x,\delta$. 
\end{theorem}

\begin{proof}
Note that property (a) of a mapping representation implies that the set $\{ \delta x: \delta \in G\}$ must include all neighbours of $x \in X$. Without loss of generality, we can assume 
$|X| \ge 2$ and $|G| \ge 2$.
Recall that irredicibility implies that the underlying combinatorial graph $X$ is connected. We first observe that $c(x,\delta)=c(\delta)$, i.e., $c(x,\delta)$ is independent of $x$. This indeed follows directly from condition (ii) in Theorem \ref{thm:EMmaprep} and connectedness.

    We next show that $c>0$.
    Suppose $c(\delta)=0$ for some $\delta \in G$. Then every $\eta \in G$ is an inverse of $\delta$ 
    in the sense of Definition \ref{def:maprep}(b). This contradicts the uniqueness requirement for inverses if $|G| \ge 2$. This implies that our maps $\delta: X \to X$ are bijective maps.
    We next show that $c(\delta)=c(\delta^{-1})$ and that $\pi$ is constant.
    From condition (c) in Definition \ref{def:maprep}, we get $c(\delta)\pi(x) = c(\delta^{-1})\pi(\delta x)$
    by choosing $F=1_{x,\delta}$.
    Hence, for all $k \in \mathbb N$
    \[
    \frac{\pi(\delta^k x)}{\pi(x)} = \left(\frac {c(\delta)}{c(\delta^{-1})} \right)^k.
    \]
    Letting $k$ go to infinity shows $c(\delta) = c(\delta^{-1})$ and thus $\pi(\delta x) = \pi(x)$ for all $\delta \in G$. By connectedness, this implies that $\pi$ is constant. 

    Now let $S := \{ \delta \in G: \delta \neq {\rm{id}_X} \}$ and $\Gamma$ be the finite group generated by $S$ via composition of bijective maps of $X$ into itself. Condition (a) in Definition \ref{def:maprep} guarantees that the set $\{ \delta x: \delta \in S \}$ contains all neighbours of $x$ and that $\Gamma$ acts transitively on $X$ by connectedness. It remains to show that $\Gamma$ acts simply transitively on $X$. Assume for $\delta \in \Gamma$ and some $x \in X$ we have $\delta x= x$. This implies $\delta \mu x = \mu \delta x = \mu x$ for all $\mu \in S$ by condition (i) in Theorem \ref{thm:EMmaprep} and, therefore, $\delta = {\rm{id}}_X$, by connectedness. Consequently, we can identify $X$ with $\Gamma$ and, as a combinational graph, it is naturally isomorphic to the abelian Cayley graph of $\Gamma$ with generating set $S$. 
\end{proof}

\medskip

In preparation for the following subsections, we recall certain facts from \cite{EM-12} for Markov chains $(X,Q,\pi)$ with mapping representations and provide some further notation. We assume the mapping representations to satisfy the extra conditions (i) and (ii) mentioned in Theorem \ref{thm:maprepabcay}.
Therefore, we have $c(x,\delta) = c(\delta)$ and $c(\delta) = c(\delta^{-1})$, and that $\pi$ is constant, so we can set $\pi \equiv \pi_0$. We have with the notation given in \cite[pages 1024-1026]{EM-12},
\begin{align}
    \mathcal{A}_\rho(f) &= \frac{1}{2} \sum_{x,\delta}  (\nabla_\delta f(x))^2 \widehat \rho(x,\delta x) c(\delta)\pi_0 , \\
    \mathcal{B}_\rho(f) &= T_1 + T_3 + T_4, \label{eq:Brho-EM}\\
    T_1 + T_3 &:= \frac{1}{4} \sum_{x,\delta,\eta} \underbrace{\left( \nabla_\delta f(\eta x) - \nabla_\delta f(x) \right)^2 \widehat \rho(x,\delta x) c(\delta)c(\eta) \pi_0}_{\ge 0}, \label{eq:T1T3-EM}\\
    T_4 &:= \frac{1}{4} \sum_{x,\delta,\eta} (\nabla_\delta f(\eta x))^2 \cdot \\ & \phantom{=} \cdot \underbrace{\left[ \widehat \rho_1(\eta x,\delta \eta x)\rho(x) + \widehat \rho_2(\eta x,\delta \eta x)\rho(\delta x) - \widehat \rho(x,\delta x) \right]}_{\text{$\ge 0$ by \eqref{eq: b-function-defn}}} c(\delta)c(\eta)\pi_0, \label{eq:T4-EM}
\end{align}
where we used the notation $\widehat \rho_i(x,y) = \partial_i \theta(\rho(x),\rho(y))$.

We also introduce, for every $\delta \in S$, the operator
\begin{align*} 
(\Delta_\delta f)(x) :=\begin{cases}
    c(\delta) ( f(\delta x) + f(\delta^{-1}x) - 2 f(x)) &\text{ if } \delta\not\in S_2, \\
    c(\delta) ( f(\delta x) - f(x)) &\text{ if } \delta\in S_2 ,
\end{cases}
\end{align*}
where $S_2 \subset S$ denote the set of all $\delta \in S$ of order $2$. We also note that \[ \Delta f(x)=\frac{1}{2}\sum_{\delta\not\in S_2}\Delta_\delta f(x)+\sum_{\delta\in S_2}\Delta_\delta f(x).\]

In the above expressions for $T_1+T_3$ and $T_4$, we 
consider only terms with $\eta$ equal to $\delta$ and $\delta^{-1}$ in the case $\delta\not\in S_2$ and terms with $\eta$ equal to $\delta$ in the case $\delta \in S_2$, and we
obtain the following lower bounds:
\begin{align}
    T_1 + T_3 &\ge \frac{1}{4} \sum_x\sum_{\delta\not\in S_2} \widehat \rho(x,\delta x) \pi_0 \cdot \nonumber \\
    &
    \cdot \left[ \underbrace{c(\delta)^2(f(\delta^2 x) + f(x) -2 f(\delta x))^2}_{= (\Delta_\delta f(\delta x))^2} + \underbrace{c(\delta)^2(2 f(x) - f(\delta x) - f(\delta^{-1} x))^2}_{(\Delta_\delta f(x))^2} \right] \nonumber \\
    &\phantom{\ge} + \frac{1}{4} \sum_x\sum_{\delta\in S_2} c(\delta)^2(2f(x)-2f(\delta x))^2 \widehat \rho(x,\delta x) \pi_0
    \nonumber \\
    &= \frac{\pi_0}{4} \sum_{x} \left( \sum_{\delta\not\in S_2}(\Delta_\delta f(x))^2 (\widehat \rho(\delta^{-1}x,x) + \widehat\rho(x,\delta x))+
    4\sum_{\delta\in S_2}(\Delta_\delta f(x))^2 \widehat\rho(x,\delta x) \right) \label{eq:T1T3}
\end{align}
and
\begin{align}
  T_4 &\ge \frac{\pi_0}{4} \sum_{x}\sum_{\delta\not\in S_2} c(\delta)^2 \bigg( (\nabla_\delta f(\delta x))^2 [\,\widehat \rho_1(\delta x,\delta^2 x)\rho(x) + \widehat \rho_2(\delta x,\delta^2 x)\rho(\delta x) - \widehat \rho(x,\delta x) \,] \nonumber \\
  &\phantom{\ge \frac{\pi_0}{4} \sum_{x}\sum_{\delta\not\in S_2} c(\delta)^2} +(\nabla_\delta f(\delta^{-1} x))^2 [\,\widehat \rho_1(\delta^{-1} x,x)\rho(x) + \widehat \rho_2(\delta^{-1} x,x)\rho(\delta x) - \widehat \rho(x,\delta x)\,]\bigg) \nonumber\\
  &\phantom{\ge}+\frac{\pi_0}{4} \sum_{x}\sum_{\delta\in S_2} c(\delta)^2(\nabla_\delta f(x))^2 [\, \widehat \rho_1(\delta x, x)\rho(x) + \widehat \rho_2(\delta x, x)\rho(\delta x) - \widehat \rho(x,\delta x) \,]\nonumber \\
  &= \frac{\pi_0}{4} \sum_{x}\sum_{\delta\not\in S_2} c(\delta)^2 \bigg(
  (\nabla_\delta f(\delta x))^2 b_\rho(x,\delta x,\delta^2 x) + (\nabla_\delta f(\delta^{-1} x))^2 b_\rho(\delta x,x,\delta^{-1}x)\bigg)\nonumber \\
  &\phantom{\ge}+\frac{\pi_0}{4} \sum_{x}\sum_{\delta\in S_2} c(\delta)^2(\nabla_\delta f(x))^2 b_\rho(x,\delta x, x)\nonumber\\
  &= \frac{\pi_0}{4} \sum_{x}\sum_{\delta\not\in S_2} c(\delta)^2 \bigg(
  (\nabla_\delta f(x))^2 b_\rho(\delta^{-1}x,x,\delta x) + (\nabla_{\delta^{-1}}f(x))^2 b_\rho(\delta x,x,\delta^{-1}x)\bigg)\nonumber\\
  &\phantom{\ge}+\frac{\pi_0}{4} \sum_{x}\sum_{\delta\in S_2} c(\delta)^2(\nabla_\delta f(x))^2 b_\rho(\delta x,x, \delta x),
  \label{eq:T4}
\end{align}
  where we used the shortened notation
  \begin{align} \label{defn: b-function-2}
      b_\rho(x,y,z) &:=
      b(\rho(x), \rho(y), \rho(z)) \\
      &\phantom{:}=\widehat \rho_1(y,z) \rho(x) + \widehat \rho_2(y,z) \rho(y) - \widehat \rho(x,y) \nonumber\\
      &\phantom{:}=\widehat \rho_1(z,y) \rho(y) + \widehat \rho_2(z,y) \rho(x) - \widehat \rho(x,y). \nonumber
  \end{align}

\subsection{(Asymptotically) sharp lower curvature bound for infinite dimension} \label{subsec:lowbdinf}

The following curvature result holds for Markov chains with mapping representations satisfying the additional conditions (i) and (ii) of Theorem \ref{thm:EMmaprep}. In view of Theorem \ref{thm:maprepabcay}, we formulate it in the equivalent setting of abelian Cayley graphs.

\begin{theorem} \label{thm:abellowbdinfdim}
    Let $(X,Q,\pi)$ be 
    an irreducible, reversible 
    finite Markov chain. Assume that the underlying combinatorial graph $X$ carries the structure of an abelian Cayley graph with generatoring set $S$ such that
    $$ Q(x,\delta x) = Q(y,\delta y) = Q(x,\delta^{-1} x) >0 \quad \text{for all $x,y \in X$ and $\delta \in S$.} $$
    Let $Q_{\rm{min}} := \min_{\delta \in S} Q(x,\delta x)$
    and $r$ be the maximal order of the generators in $S$. Then $(X,Q,\pi)$ satisfies $CD_{ent}(K,\infty)$ with
    \[
    K = \frac{Q_{\rm{min}}}{50 r^4}.
    \]
\end{theorem}

\begin{proof}
Note that we have $r \ge 2$, since we assume that our group is non-trivial and that $Q_{\min} \le c(\delta)$ for all $\delta \in S$.
For $\delta \in S$, we partition $X$ into disjoint $\delta$-orbits $\mathcal{O}_x := \delta^\mathbb{Z} x$.
The set of all $\delta$-orbits is denoted by $X / \sim_\delta$ with the canonically associated equivalence relation $\sim_\delta$ on $X$. For $\mathcal{O} \subset X$, we define 
\[A(\delta, \mathcal{O}) 
:=\frac{c(\delta)\pi_0}{2} \sum_{x \in \mathcal{O}} \widehat \rho(x,\delta x) (\nabla_\delta f(x))^2,
 \]
and we define $B(\delta,\mathcal{O})$ according to the order of $\delta$ as follows. In the case $\delta\not\in S_2$, we define
\begin{align} \label{eq:B-delta-O}
B(\delta,\mathcal{O}) :=&\frac{\pi_0}{4}\sum_{x \in \mathcal{O}} \left[
    (\Delta_\delta f(x))^2(\widehat\rho(x,\delta x) + \widehat\rho(x,\delta^{-1} x)) \phantom{\bigg( \bigg)}\right. \\& 
    + \left. c(\delta)^2 \bigg(
  (\nabla_\delta f(x))^2 b_\rho(\delta^{-1}x,x,\delta x) + (\nabla_{\delta^{-1}} f(x))^2 b_\rho(\delta x,x,\delta^{-1}x)\bigg)\right].
  \nonumber
\end{align}
In the case of $\delta\in S_2$, we define
\begin{align} \label{eq:B-delta-O-2}
B(\delta,\mathcal{O}) :=\frac{c(\delta)^2 \pi_0}{4}\sum_{x \in \mathcal{O}} 
    4(\nabla_\delta f(x))^2\widehat\rho(x,\delta x) +
  (\nabla_\delta f(x))^2 b_\rho(\delta x,x,\delta x).
\end{align}
Then we have
\[
\mathcal A_\rho(f) = \sum_{\delta \in S} \sum_{\mathcal{O} \in X/\sim_\delta} A(\delta,\mathcal{O}),
\]
and
\[
\mathcal B_\rho(f) \geq \sum_{\delta \in S} \sum_{\mathcal{O} \in X/\sim_\delta} B(\delta,\mathcal{O}).
\]

Let $K:=0.02Q_{\min}r^{-4}$. In order to show that $K_{ent}(X,Q,\pi) \ge K$, it suffices to show that $B(\delta,\mathcal{O}) \ge 0.02c(\delta)r^{-4}A(\delta,\mathcal{O})$ for every $\delta\in S$ and $O\in X/\sim_\delta$. 

When $\delta\in S_2$ (which means $\mathcal{O}=\{x,\delta x\}$ for some $x$), the terms $A(\delta,\mathcal{O})$ and $B(\delta,\mathcal{O})$ are simplified to
\begin{align*}
    A(\delta,\mathcal{O}) &=c(\delta)\pi_0 \widehat \rho(x,\delta x) (\nabla_\delta f(x))^2, \\
    B(\delta,\mathcal{O}) &= 2c(\delta)A(\delta,\mathcal{O}) + \frac{c(\delta)^2\pi_0}{4}
  (\nabla_\delta f(x))^2 \Big(b_\rho(\delta x,x,\delta x)+b_\rho(x,\delta x,x)\Big) ,
\end{align*}
which implies $B(\delta,\mathcal{O})\ge 2c(\delta)A(\delta,\mathcal{O})$.

\smallskip

It remains to consider the case when $\delta\not\in S_2$. Let $m \in \mathcal{O}$ be
such that
\[
\widehat \rho(m,\delta m)(\nabla_\delta f(m))^2
\]
is maximal, and set $\rho_0:=\widehat \rho(m,\delta m)$. We can assume $\rho_0 >0$ and $\nabla_\delta f(m) \neq 0$ for, otherwise, $A(\delta,\mathcal{O})=0$ and there is nothing to prove. Since $|\mathcal{O}|\le r$, it follows that 
\begin{align} \label{eq:A-delta-O_est}
    A(\delta,\mathcal{O}) \le \frac{1}{2} rc(\delta) \pi_0\rho_0(\nabla_\delta f(m))^2. 
\end{align}

Note that at least one of the vertices in $\{m,\delta m\}$ satisfies $\rho(v) \ge \rho_0$. Let $I=(\delta^j m,\delta^{j+1} m,\ldots, \delta^k m)$ with $j \le k$ be the maximal injective sequence such that it contains at least one of $\{m,\delta m\}$ and that $\rho(x) \geq \dfrac{1}{e^2}\rho_0$ for all $x \in I$.
To estimate the term $B(\delta,\mathcal{O})$, we consider the following two separated cases.

\medskip

\textbf{Case 1:} There exists $x_0\in \{\delta^{j-1}m\}\cup I$ such that $|\nabla_\delta f(x_0)-\nabla_\delta f(m)| > \frac{1}{2}|\nabla_\delta f(m)|$.

\smallskip

\noindent

Recall that $\dfrac{1}{c(\delta)}\Delta_\delta f(x)=f(\delta x) + f(\delta^{-1}x) - 2f(x)=\nabla_\delta f(x)-\nabla_{\delta} f(\delta^{-1}x)$. Restricting to some subinterval $J= (v,\delta v,\ldots, \delta^{-1}w, w)$ of $I$ gives
\[
\sum_{x \in I} (\Delta_\delta f(x))^2 \ge \frac{1}{|J|} \left(\sum_{x\in J} \Delta_\delta f(x) \right)^2 \ge  \frac{c(\delta)^2}{r} \left(\nabla_\delta f(w) - \nabla_{\delta} f(\delta^{-1}v) \right)^2.
\]
In other words, the inequality
\[
\sum_{x \in I} (\Delta_\delta f(x))^2 \ge  \frac{c(\delta)^2}{r} \left(\nabla_\delta f(y) - \nabla_{\delta} f(z) \right)^2
\] 
holds true for any two edges $\{y,\delta y\}$ and $\{z, \delta z\}$ with each having at least one vertex in $I$.
Applying the above inequality with $y = x_0$ and $z=m$ to the expression \eqref{eq:B-delta-O} and using the fact that $\widehat\rho(x,\delta x) + \widehat\rho(x,\delta^{-1} x) \geq \frac{1}{e^2}\rho_0$ for all $x\in I$, we obtain
\begin{align*}
    B(\delta,\mathcal{O}) &\ge \frac{\pi_0}{4}\sum_{x \in \mathcal{O}} 
    (\Delta_\delta f(x))^2(\widehat\rho(x,\delta x) + \widehat\rho(x,\delta^{-1} x)) \\
    &\ge \frac{c(\delta)^2\pi_0\rho_0}{16e^2r}(\nabla_\delta f(m))^2.
\end{align*}
Together with \eqref{eq:A-delta-O_est}, we have
\[
\frac{B(\delta,\mathcal{O})}{A(\delta,\mathcal{O})} \ge \frac{c(\delta)}{8e^2r^2} > \frac{0.02c(\delta)}{r^4}.
\]

\smallskip

\textbf{Case 2:} Every $x\in \{\delta^{j-1}m\}\cup I$ satisfies $|\nabla_\delta f(x)-\nabla_\delta f(m)| \le \frac{1}{2}|\nabla_\delta f(m)|$.

This condition implies that, for all $x\in I$, we have
\[
|\nabla_\delta f(x)|\ge \frac{1}{2}|\nabla_\delta f(m)| \quad \text{and} \quad |\nabla_{\delta^{-1}} f(x)|\ge \frac{1}{2}|\nabla_\delta f(m)|\,.
\]
We can then estimate $B(\delta,\mathcal{O})$ by using \eqref{eq:B-delta-O}, Proposition \ref{prop: b-function}(d), and the assumption that $\rho(x)\ge \frac{1}{e^2}\rho_0$ for all $x\in I$:
\begin{align}
    B(\delta,\mathcal{O}) &\geq \frac{\pi_0c(\delta)^2(\nabla_\delta f(m))^2}{16}\sum_{x \in I}  (b_\rho(\delta^{-1}x,x,\delta x) +  b_\rho(\delta x,x,\delta^{-1}x)) \nonumber \\
    &\geq \frac{\pi_0c(\delta)^2(\nabla_\delta f(m))^2}{200}\sum_{x \in I} \rho(x) (\log \rho(\delta^{-1}x)+\log \rho(\delta x)-2\log \rho(x))^2 \nonumber \\
    &\ge \frac{\pi_0\rho_0}{200e^2}(\nabla_\delta f(m))^2 \sum_{x \in I}  (\Delta_\delta \log \rho(x))^2. \label{eq:BdelOlower}
\end{align}
Furthermore, the condition in this case implies that $\nabla_\delta f(x)$ takes the same sign as $\nabla_\delta f(m)$ for all $x\in I$. In particular, $\sum_{x\in I} \nabla_\delta f(x)\not=0=\sum_{x\in\mathcal{O}} \nabla_\delta f(x)$. Thus $I\not=\mathcal{O}$.

Now consider the outer boundary points of $I$, namely $y=\delta^{j-1}m$ and $z=\delta^{k+1}m$ (with possibility to coincide).
Then we know by definition that $\rho(y),\rho(z)< \frac{1}{e^2}\rho_0$. Moreover $\rho(u)\geq \rho_0$ for some $u\ \in I \cap \{ m,\delta m\}$. Let $u=\delta^{\ell} m$. Hence there exist $v = \delta^\alpha m$ and $w=\delta^\beta m$ in $I$ (with $\alpha\le\ell\le\beta$) such that
\begin{align*}
\nabla_\delta \log \rho (\delta^{-1}v) \ge \frac{\log (e^2)}{r}, \quad  \text{and} \quad
\nabla_\delta \log \rho (w) \le -\frac{\log (e^2)}{r}.
\end{align*} 
Thus,
\begin{align*}
\sum_{x \in I} \left( \Delta_\delta \log\rho(x) \right)^2 &\ge \frac{1}{|\beta-\alpha+1|}\left(\sum_{i=\alpha}^\beta (\Delta_\delta \log \rho)(\delta^i m)\right)^2 \\ 
&\ge \frac{c(\delta)^2}{r}\left(\nabla_\delta \log \rho (w) -  \nabla_{\delta} \log \rho (\delta^{-1}v)\right)^2 \ge \frac{16c(\delta)^2}{r^3}.
\end{align*}
Applying this to \eqref{eq:BdelOlower} gives
\[
     B(\delta,\mathcal{O}) \ge \frac{0.08\pi_0\rho_0c(\delta)^2}{e^2r^3}(\nabla_\delta f(m))^2 .
\]
Together with \eqref{eq:A-delta-O_est}, we have
\begin{align*}
    \frac{B(\delta,\mathcal{O})}{A(\delta,\mathcal{O})} \ge \frac{0.02c(\delta)}{r^4}.
\end{align*}

\end{proof}

\subsection{Lower curvature bound for finite dimension}
\label{subsec:lowbdfin}

In this subsection, we derive a finite dimensional lower curvature bound for Markov chains with underlying abelian Cayley graph structure, where the dimension is expressed in terms of the cardinalities of the symmetric set of generators and its subset of generators of order $2$.



\begin{theorem} \label{thm:abellowbdfindim}
    Let $(X,Q,\pi)$ be 
    an irreducible, reversible finite Markov chain. Assume that the underlying combinatorial graph $X$ carries the structure of an abelian Cayley graph with generating set $S$ such that
    $$ Q(x,\delta x) = Q(y,\delta y) = Q(x,\delta^{-1} x) >0 \quad \text{for all $x,y \in X$ and $\delta \in S$.} $$
    Then $(X,Q,\pi)$ satisfies $CD_{ent}(0,N)$ with $N=0.93|S|+0.07|S_2|$, where $S_2 \subset S$ are the generators of order $2$. 
\end{theorem}

\begin{remark} \label{rem:comBEent}
    The proof of Theorem \ref{thm:abellowbdfindim} can be adapted to Bakry-\'Emery curvature with the only difference being $N=|S|$ (that is, using $b \equiv 0$ in the case $\theta= \theta_a$ in the proof). This means an abelian Cayley graph with generating set $S$ always satisfies the Bakry-\'Emery curvature condition $CD_a(0,|S|)$. In the special case that all reduced relations of the generators except for the commutator relations have length $\ge 5$, this fact is also confirmed by Theorem 9.1 in \cite{CLP-19}. Furthermore, abelian Cayley graphs are always Ricci flat graphs but not vice versa (see \cite{LL-24} for more details). Ricci flat graphs were introduced by \cite{CY-96} and they are $d$-regular graphs satisfying also $CD_a(0,d)$, as mentioned in \cite{munch2018li}. Moreover, it was shown in \cite{bauer2015li} that Ricci flat graphs satisfy $CDE(0,d)$.  
\end{remark}

\begin{remark}
    In the case when an abelian Cayley graph has only generators of order $2$ (i.e., $S=S_2$), it satisfies the entropic curvature condition $CD_{ent}(0,|S|)$. This dimension $|S|$ is sharp in the case of, e.g., the simple random walk on the hypercube $(\mathbb{Z}_2)^d$ (check with $\rho\equiv 1$ and $f(x)\in\{0,1\}$ according to the bipartiteness). In this case we have both $CD_{ent}(0,d)$ and $CD_a(0,d)$.

    In the case when an abelian Cayley graph has some generator of higher order, it satisfies a stronger curvature-dimension condition, namely, $CD_{ent}(0,N)$ with $N=0.93|S|+0.07|S_2|<|S|$.
\end{remark}


\begin{proof}
Recall that $(X,Q,\pi)$ is $CD_{ent}(0,N)$ if 
$$ \mathcal{B}_\rho(f) \ge \frac{1}{N}  \langle
\rho, (\Delta f)^2 \rangle_{\pi}, $$
holds true for all $f$ and $\rho$. We are now using the estimates derived at the end of Subsection \ref{subsec:specmap}. We have $\mathcal{B}_\rho(f) = (T_1 + T_3) + T_4$
with (see \eqref{eq:T1T3})
$$ T_1 + T_3 \ge \frac{\pi_0}{4} \sum_{x} \left( \sum_{\delta\not\in S_2}(\Delta_\delta f(x))^2 (\widehat \rho(\delta^{-1}x,x) + \widehat\rho(x,\delta x))+
    4\sum_{\delta\in S_2}(\Delta_\delta f(x))^2 \widehat\rho(x,\delta x) \right).
$$
Regarding $T_4$, we use \eqref{eq:T4} and the inequality
$$ c_1 x_1^2 + c_2 x_2^2 \ge \frac{c_1c_2}{c_1+c_2}(x_1+x_2)^2, $$
and obtain
\begin{align*}
  T_4 &\ge \frac{\pi_0}{4} \sum_{x}\sum_{\delta\not\in S_2} c(\delta)^2 \bigg(
  (f(\delta x) - f(x))^2 b_\rho(\delta^{-1}x,x,\delta x) + (f(\delta^{-1} x) - f(x))^2 b_\rho(\delta x,x,\delta^{-1}x)\bigg)\\
  &\phantom{\ge}+\frac{\pi_0}{4} \sum_{x}\sum_{\delta\in S_2} c(\delta)^2(f(\delta x)-f(x))^2 b_\rho(\delta x,x, \delta x)\\
  &\ge \frac{\pi_0}{4} \sum_{x}\left(\sum_{\delta\not\in S_2} (\Delta_\delta f(x))^2 \frac{b_\rho(\delta^{-1}x,x,\delta x) \cdot b_\rho(\delta x,x,\delta^{-1}x)}{b_\rho(\delta^{-1}x,x,\delta x)+b_\rho(\delta x,x,\delta^{-1}x)} +\sum_{\delta\in S_2} (\Delta_\delta f(x))^2 b_\rho(\delta x,x, \delta x)\right).
\end{align*}

Bringing everything together, applying Proposition \ref{prop: b-function}(f) and (g), and then applying Cauchy-Schwarz, we finally obtain
\begin{align*}
\mathcal{B}_\rho(f) &= T_1 + T_3 + T_4 \\
&\ge \frac{\pi_0}{4} \sum_{x}\left(\sum_{\delta\not\in S_2} (\Delta_\delta f(x))^2 \cdot \frac{1}{C}\rho(x) +\sum_{\delta\in S_2} (\Delta_\delta f(x))^2 \cdot 2(\rho(x)+\rho(\delta x))\right)\\
&= \frac{\pi_0}{4} \sum_x \rho(x) \left( \sum_{\delta\not\in S_2} \frac{1}{C}(\Delta_\delta f(x))^2+\sum_{\delta\in S_2} 4(\Delta_\delta f(x))^2 \right) \\
&\ge \frac{\pi_0}{4} \sum_x \rho(x)\frac{\left( \sum_{\delta\not\in S_2} \Delta_\delta f(x) +2 \sum_{\delta\in S_2} \Delta_\delta f(x)\right)^2}{C|S-S_2|+|S_2|}\\
&= \frac{1}{C|S|+(1-C)|S_2|} \langle
\rho, (\Delta f)^2 \rangle_{\pi},
\end{align*}
for $C=0.93$. Consequently, we can conclude that $(X,Q,\pi)$ satisfies $CD_{ent}(0,N)$ with 
$$ N\le 0.93|S|+0.07|S_2|. $$
\end{proof}

\subsection{(Asymptotically) sharp upper curvature bound for cycles} \label{subsec: uppbd-cycle}

Let $C_n$ be the cycle of length $n$, equipped with the simple random walk $Q(x,y) = \frac{1}{2}$ for adjacent vertices $x,y$.
In this subsection, we improve the trivial upper bound $K_{ent} \leq \lambda_1 = 2\sin^2\left(\frac{\pi}{n}\right) \le \frac{2\pi^2}{n^2}$, derived from the Lichnerowicz inequality (see, e.g., \cite[Theorem 1.3]{KLMP-23}) to $K\lesssim n^{-4} \log(n)^2$.

\begin{theorem} \label{thm:circleentbd}
    Let $X$ be a cycle graph on $4n$ vertices with $Q(x,y)=q \le \frac{1}{2}$ for neighbors $x\sim y$. Assume $n\geq 4$. Then
    \[
    K_{ent} \leq C qn^{-4}(\log n)^2,
    \]
    where $C=5000$.
\end{theorem}

\begin{proof}
    For simplicity, we drop the restriction of $\rho$ to be a probability measure (note that our considerations are invariant under rescaling, which allows this generalization). 
    Let the vertices of the cycle $X=\mathbb{Z}/(4n\mathbb{Z})=\{-2n+1,\ldots,2n\}$. We construct a function $f\in \mathbb{R}^X$ as
    \begin{align*}
        f(x) &:= \begin{cases} k & \text{if $-n \le x \le n$}\,, \\
    2n-x & \text{if $n \le x \le 2n$}\,, \\ -2n-x & \text{if $-2n+1 \le x \le -n$}\,, \end{cases}
    \end{align*} 
    and a measure $\rho\in \mathcal{P}_*(X)$ as
    \[
    \rho(x):=\exp(-4(x/n)^2 \cdot \log(n)) \quad \text{for all $-2n <x \le 2n$}\,.
    \]
    
    Given that $n\ge 4$, we make the following three observations:
    \begin{enumerate}[(i)]
        \item $\rho(n-1) \le 16n^{-4}$. 
        \item $\dfrac{\rho(x)}{\rho(x+1)} \in [\frac{1}{256},256]$. 
        \item $\log\left(\frac{\rho(x-1) \rho(x+1)}{\rho(x)^2}\right)=\frac{-8\log(n)}{n^2}$ for all $x\not=2n$, and \\
        $\log\left(\frac{\rho(2n-1) \rho(2n+1)}{\rho(2n)^2}\right) =\frac{4\log(n)(8n-2)}{n^2}\le 11$.
    \end{enumerate}

Now recall from \eqref{eq:Brho-EM}, \eqref{eq:T1T3-EM} and \eqref{eq:T4-EM} that
\begin{align*}
    \mathcal B_\rho(f) &= T_1+T_3+T_4 , \\
    T_1+T_3 &=\frac{\pi_0}{2} \sum_{x} 
    (\Delta f(x))^2(\widehat\rho(x,\delta x) + \widehat\rho(x,\delta^{-1} x)) ,\\
    T_4 &= \frac{\pi_0q^2}{2}\sum_{x}
  (\nabla_\delta f(x))^2 b_\rho(\delta^{-1}x,x,\delta x) + (\nabla_{\delta^{-1}}f(x))^2 b_\rho(\delta x,x,\delta^{-1}x) ,
\end{align*}
where $\delta x = x+1$ and $\delta^{-1} x = x-1$.

For the estimate of $T_1+T_3$, we notice that $|\Delta f(\pm n)|=2q$ and $\Delta f(x) = 0$ for $x \notin \{\pm n\}$. Thus
\begin{align*}
T_1+T_3 &= \frac{\pi_0(2q)^2}{2} \Big(\widehat \rho(-n,-n-1) + \widehat \rho(-n,-n+1) +
\widehat \rho(n,n-1) + \widehat \rho(n,n+1) \Big) \\
&\leq 8\pi_0q^2 \rho(n-1) \leq 128\pi_0q^2n^{-4} .
\end{align*}
For $T_4$, we notice that $|\nabla_\delta f(x)|=|\nabla_{\delta^{-1}} f(x)|=1$ and
that $\frac{\rho(x)}{\rho(x+1)} \in [\frac{1}{256},256]$ for all vertices $x\in X$. Hence, applying Proposition \ref{prop: b-function}(e), we have
\begin{align*}
    T_4 &\le \frac{\pi_0q^2}{2}\sum_{x} 34 
  \rho(x)\left(\log\left(\frac{\rho(x-1) \rho(x+1)}{\rho(x)^2} \right) \right)^2 \\
  &\le \frac{\pi_0q^2}{2}\cdot 34  \left( 64 n^{-4}(\log n)^2 \sum_{x \neq 2n} \rho(x) + 121 \rho(2n) \right) \\
  &\le 1089\pi_0q^2n^{-4}(\log n)^2 \sum_x \rho(x).
\end{align*}
Combining the above two estimates, we obtain
\begin{align*}
    \mathcal B_\rho(f) = T_1+T_3+T_4 \le 1217\pi_0q^2n^{-4}(\log n)^2\sum_x \rho(x).
\end{align*}


On the other hand, using the formula \eqref{eq:Gammarho}, we have for all $x$,
\[
\Gamma_\rho f (x) =  q\left(\partial_1\theta(\rho(x),\rho(x+1)) + \partial_1\theta(\rho(x),\rho(x-1)) \right) \geq 2q\partial_1\theta(256,1) \geq 0.29q,
\]
and therefore, 
\[
\mathcal A_\rho(f) = \langle \rho,\Gamma_\rho f \rangle \geq 0.29\pi_0q \sum_x \rho(x).
\]
Combining with the estimate for $\mathcal B_\rho(f)$, we obtain
\begin{align*}
    K_{ent}(X) \leq \frac{{\mathcal B}_\rho(f)}{{\mathcal A}_\rho(f)} \leq 5000q n^{-4}(\log n)^2.
\end{align*}
This finishes the proof of the theorem.
\end{proof}

\subsection{Abelian Cayley graphs and Ricci flatness}

Abelian Cayley graphs belong to the class of Ricci flat graphs, which were introduced in \cite{CY-96}. The following definition presents special versions of this notion:

\begin{definition}[see {\cite[Def. 3.1]{cushing2021curvatures}}] \label{def:ricciflat}
Let $G=(V,E)$ be a $d$-regular graph. A vertex $x \in V$ is \emph{Ricci flat} if there exist maps $\eta_j: B_1(x) \to V$, $1 \le j \le d$, with the following properties:
\begin{itemize}
\item[(i)] $\eta_j(y) \sim y$ for all $y \in B_1(x)$,
\item[(ii)] $\eta_i(y)\neq \eta_j(y)$ for all $y \in B_1(x)$ and $i \neq j$,
\item[(iii)] $\bigcup_{j} \eta_j(\eta_i(x)) = \bigcup_{j} \eta_i(\eta_j(x))$ for all $i$.
\end{itemize}
The additional properties
\begin{align*}
&(R)& \quad \eta_i^2(x) &= x \quad \text{for all $i$,} \\
&(S)& \quad \eta_j(\eta_i(x)) &= \eta_i(\eta_j(x)) \quad \text{for all $i,j$,}
\end{align*}
are called \emph{reflexivity} and \emph{symmetry}, respectively. 

If there exists a family of maps $\eta_j$ for a given vertex $x \in V$ satisfying (R) or (S) in addition to (i)--(iii), we say that $x$ is $(R)$-Ricci flat or $(S)$-Ricci flat, respectively. If there exists a family of maps $\eta_j$ satisfying (i)--(iii) and (R) and (S) simultaneously, we say that $x$ is $(RS)$-Ricci flat. 
\end{definition}

The (S)-Ricci flatness definition is similar in spirit to the special mapping representation considered in Theorem \ref{thm:EMmaprep}
with the main difference that Ricci flatness is a local property. It is known that Ricci flatness implies both non-negative Ollivier Ricci curvature and non-negative Bakry-\'Emery curvature (see, e.g., \cite{cushing2021curvatures} and the references therein).
It is not difficult to see that Ricci flatness also implies non-negative Ollivier sectional curvature (see Definition \ref{def:OllRicsec} below).

It would be interesting to investigate whether Ricci flatness also implies non-negative entropic curvature. Let us discuss some potential naive strategies to tackle this question.

Note that all abelian Cayley graphs are (S)-Ricci flat, but not vice versa. Examples of such graphs can be found in \cite{LL-24}. Optimistic readers may hope that every Ricci flat graph has an abelian Cayley graph as a covering. A sujective graph homomorphism $\pi: \tilde G \to G$ is called a covering if it is locally bijective, that is 
one-balls $B_1(x)$ of $\tilde G$ are bijectively mapped onto one-balls $B_1(\pi(x))$ of $G$. 
A counterexample for the above hope is the graph in Figure \ref{fig:C4C5comp} (the complement of $C_4 \dot\cup C_5$) together with Lemma \ref{lem:liftRicci}. Since this graph is highly connected, we suspect that it has positive entropic curvature.

\begin{figure}[h!]
\begin{center}
\includegraphics[width=8cm]{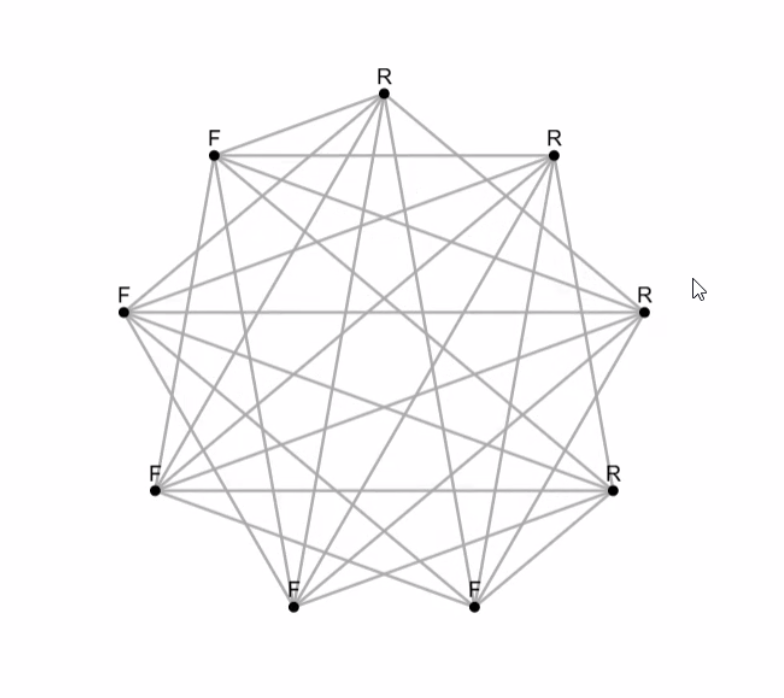}
\end{center}
\caption{A Ricci flat graph which is not (S)-Ricci flat, and therefore not an abelian Cayley graph (vertices labelled ``F'' and ``R'' are only Ricci-flat and only (R)-Ricci flat, respectively). The picture was generated via the Graph Curvature Calculator \cite{CKLLS-22}.}
\label{fig:C4C5comp}
\end{figure}

\begin{lemma}
\label{lem:liftRicci} Let $\pi: \tilde G \to G$ be a covering. Let $x$ be a vertex in $G$ and $\tilde x \in \pi^{-1}(\{x\}) \in \tilde G$. If $\tilde x$ is (S)-Ricci flat in $\tilde G$, then $x$ is also (S)-Ricci flat in $G$. 
\end{lemma}

\begin{proof}
Let $\tilde G = (\tilde V, \tilde E)$ and $G = (V,E)$. Assume that $\tilde x \in \tilde V$ is $(S)$-Ricci flat. Then $\pi\vert_{B_1(\tilde x)}: B_1(\tilde x) \to B_1(x)$ is bijective with inverse $\rho: B_1(x) \to B_1(\tilde x)$. Let $\tilde \eta_j: B_1(\tilde x) \to \tilde V$ be as in Definition \ref{def:ricciflat} and 
$$ \eta_j = \pi \circ \tilde \eta_j \circ \rho: B_1(x) \to V. $$

Since both $\rho$ and $\pi$ are graph homomorphisms and $\tilde \eta_j(\tilde y) \sim \tilde y$ for all $\tilde y \in B_1(\tilde x)$, we have $\eta_j(y) \sim y$ for all $y \in B_1(x)$. This shows (i).

Property (ii) follows from the local bijectivity of $\pi$ around $\rho(y)$.

Property (S) (and therefore also (iii)) follows via the following identities:

\begin{align*} 
\eta_j(\eta_i(x)) &= \pi \circ \tilde \eta_j \circ (\rho \circ \pi) \circ \tilde \eta_i \circ \rho(x) \\
&= \pi \circ \tilde \eta_j \circ \tilde \eta_i \circ \rho(x) \\
&= \pi \circ \tilde \eta_i \circ \tilde \eta_j \circ \rho(x) \\
&= \eta_i(\eta_j(x)).
\end{align*}
\end{proof}

\section{Universal lower curvature bounds and perturbation} \label{sec:universal_bd_perturbed}

In this section, we derive explicit lower curvature bounds for general finite Markov chains and present perturbation results. It was shown in \cite[Theorem 4.1]{Mi-13} that the entropic curvature of any finite Markov chain is always greater than $- \infty$. Our explicit lower curvature bounds are given in Theorem \ref{thm:univgenlowbondinfdim} for infinite dimension and Theorem \ref{thm:univgenlowbondfindim} for finite dimension. While the second result is based on an inequality for the logarithmic mean (Lemma \ref{lem:ineqforgenlowbd}) and holds therefore only for entropic curvature, the first one is a lower curvature bound for arbitrary means. These explicit results inspired also the derivation of a lower bound for the intertwining curvature given in \cite{MWZ-24}. The perturbation results are concerned with comparisons of two Markov chains with closeby parameters and closeby means: Theorem \ref{thm:perturb1} compares so-called \emph{$\delta$-curvatures} of two Markov chains with the same graph structure and bounds on the $\log$-differences of their steady states, transition probabilities and means. Theorem \ref{thm:perturb2} assumes that both Markov chains are paired with the same mean $\theta$ and provides a bound on the difference between the \emph{$\theta$-curvatures} of two weighted graphs with the same topology and bounds on the $\log$-differences of the vertex measures and edge weights. All results are given in terms of the minimum $Q_{\min}$ of the positive transition probabilities of the involved Markov chains. 


\subsection{Lower curvature bound for infinite dimension}
\label{subsec:lowcurvbdinfdim}

We start with a lemma which will be also important for later perturbation considerations. In order to state the lemma, we first need some preparations.

For a given finite Markov chain $(X,Q,\pi)$, we define for any $x,y,z \in X$,
\begin{equation} \label{eq:qxyz}
q_{xyz} := Q(y,x)Q(y,z)\pi(y) \ge 0. 
\end{equation}
Then we can write $\mathcal{A}_\rho(f)$ and $\mathcal{A}_\rho(f,\Delta f)$ in the following way:
\begin{align}
    \mathcal{A}_\rho(f) &= \Vert \nabla f \Vert_\rho^2 = \frac{1}{2} \sum_{x,y} \widehat \rho_{xy}(f(y)-f(x))^2 Q(x,y) \pi(x) \nonumber \\
    &= \frac{1}{2} \sum_{x,y,z} \widehat \rho_{xy} Q(y,x) Q(y,z) \pi(y) (f(y)-f(x))^2 \nonumber \\ 
    &= \frac{1}{2} \sum_{x,y,z} q_{xyz} \widehat \rho_{xy} (f(y)-f(x))^2 \nonumber \\
    &= \frac{1}{4} \sum_{x,y,z} q_{xyz} \left( \widehat \rho_{xy} (f(y)-f(x))^2 + \widehat \rho_{zy}(f(y)-f(z))^2 \right)
    \label{eq:Arhoqxyz}
\end{align}
and
\begin{align}
    \mathcal{A}_\rho(f,\Delta f) &= \langle \nabla f,\nabla \Delta f \rangle_\rho = \frac{1}{2} \sum_{x,y} \widehat \rho_{xy}(f(y)-f(x))(\Delta f(y)-\Delta f(x)) Q(x,y) \pi(x) \nonumber \\
    &= -\sum_{x,y,z} q_{xyz} \widehat \rho_{xy} (f(y)-f(x))(f(y)-f(z)) \nonumber \\
    &= - \frac{1}{2} \sum_{x,y,z} q_{xyz} (\widehat \rho_{xy} + \widehat \rho_{zy}) (f(y)-f(x)) (f(y)-f(z)). \label{eq:ArhofDfqxyz}
\end{align}
For $x \in \mathbb{R}$, let $[x]^{\pm} \ge 0$ be the positive/negative part of $x$, respectively, and therefore, $|x| = [x]^+ + [x]^-$. In view of \eqref{eq:ArhofDfqxyz}, we define
\begin{equation} \label{eq:Arhopm}
[\mathcal{A}_\rho]^{\pm}(f,\Delta f) := \sum_{x,y,z} \left[ - \frac{1}{2} q_{xyz} (\widehat \rho_{xy} + \widehat \rho_{zy}) (f(y)-f(x)) (f(y)-f(z)) \right]^{\pm}. 
\end{equation}
In the following lemma, we use the notation
\begin{align} 
| \mathcal{A}_\rho |(f,\Delta f) &:= 
[\mathcal{A}_\rho]^+(f,\Delta f) + [\mathcal{A}_\rho]^-(f,\Delta f) \nonumber \\ &=
\frac{1}{2} \sum_{x,y,z} q_{xyz} (\widehat \rho_{xy} + \widehat \rho_{zy}) \cdot |f(y)-f(x)|\cdot |f(y)-f(z)|. \label{eq:absArhofDf}
\end{align}
Finally, we define $P = \Delta + \rm{Id}$, which can be written as
\begin{equation} \label{eq:Pdef} 
P \rho(x) = \sum_y Q(x,y) \rho(y). 
\end{equation}

\begin{lemma} \label{lem:univlowbduseful}
    Let $(X,Q,\pi)$ be an irreducible, reversible finite Markov chain, $Q_{\min} := \min_{x \sim y} Q(x,y)$, $\theta$ an arbitrary mean, and $C > 0$. Then we have for all $\rho \in \mathcal{P}_*(X)$ and all $f \in \mathbb{R}^X$,
    \begin{equation} \label{eq:infcurvlowbd}
      \frac{1}{2} \langle P \rho, \Gamma_\rho f\rangle_\pi + \frac{2C^2}{Q_{\min}}\mathcal{A}_\rho(f) \ge C |\mathcal{A}_\rho|(f,\Delta f)
    \end{equation}
    with $P \rho$ given in \eqref{eq:Pdef} and $|\mathcal{A}_\rho|(f,\Delta f)$ given in \eqref{eq:absArhofDf}.
\end{lemma}

For the proof and for later considerations, we introduce the following notation in accordance with \eqref{eq: b-function-defn}:
\begin{equation} \label{defn: b0-function}
b_0(\alpha,\beta,\gamma):=\partial_1\theta(\beta,\gamma)\alpha + \partial_2\theta(\beta,\gamma)\beta.
\end{equation}

\begin{proof}
  Using \eqref{eq:Pdef} and the fact that $P \rho(x) \ge Q_{\min} \rho(y)$ for all $y \sim x$, we obtain, via symmetrization,
  \begin{align*}
      \langle P \rho,\Gamma_\rho f \rangle_\pi &= \sum_y P\rho(y) \Gamma_\rho f(y) \pi (y) \\
      &= \sum_y P\rho(y) \left( \sum_x \partial_1 \theta(\rho_y,\rho_x) (f(y)-f(x))^2 Q(y,x) \right) \pi(y) \\
      &= \frac{1}{2} \sum_{x,y} (f(y)-f(x))^2 (P\rho(y) \partial_1 \theta(\rho_y,\rho_x) + P\rho(x) \partial_1 \theta(\rho_x,\rho_y))Q(y,x)\pi(y) \\
      &\ge \frac{1}{2} \sum_{x,y,z} q_{xyz} (f(y)-f(x))^2 (\partial_1 \theta(\rho_y,\rho_x)\rho_z + Q_{\min} \partial_1 \theta(\rho_x,\rho_y) \rho_y) \\
      &\ge \frac{Q_{\min}}{2} \sum_{x,y,z} q_{xyz} (f(y)-f(x))^2 (\partial_1 \theta(\rho_y,\rho_x)\rho_z + \partial_2 \theta(\rho_y,\rho_x) \rho_y) \\
      &= \frac{Q_{\min}}{2} \sum_{x,y,z} q_{xyz} (f(y)-f(x))^2 b_0(\rho_z,\rho_y,\rho_x).
  \end{align*}
  A second symmetrization yields
  \begin{multline} \label{eq:PrhoGammarhofnorm}
  \langle P \rho,\Gamma_\rho f \rangle_\pi \ge \\ \frac{Q_{\min}}{4} \sum_{x,y,z} q_{xyz} \left( (f(y)-f(x))^2 b_0(\rho_z,\rho_y,\rho_x) + (f(y)-f(z))^2 b_0(\rho_x,\rho_y,\rho_z) \right). 
  \end{multline}
  Since (see \cite[Lemma 2.2]{EM-12})
  \begin{equation} \label{eq:b0est0} 
  b_0(\rho_z,\rho_y,\rho_x) = \partial_1 \theta(\rho_y,\rho_x)\rho_z + \partial_2 \theta(\rho_y,\rho_x)\rho_y \ge \widehat \rho_{yz}, 
  \end{equation}
  we finally obtain the inequality
  \begin{equation} \label{eq:PrhoGrhof}
  \langle P \rho,\Gamma_\rho f \rangle_\pi \ge \frac{Q_{\min}}{4} \sum_{x,y,z} q_{xyz} (f(y)-f(x))^2 \widehat \rho_{yz} + (f(y)-f(z))^2 \widehat \rho_{yx}).
  \end{equation}
  Using the expressions \eqref{eq:PrhoGrhof}, \eqref{eq:Arhoqxyz} and \eqref{eq:absArhofDf}, the inequality of the lemma follows now via the termwise inequalities
  $$ \frac{Q_{\min}}{8}\left( A^2 \widehat \rho_{zy} + B^2 \widehat \rho_{xy} \right) + \frac{C^2}{2Q_{\min}} \left( A^2 \widehat \rho_{xy} + B^2 \widehat \rho_{zy} \right) \ge \frac{C}{2} (\widehat \rho_{xy} + \widehat \rho_{zy}) \cdot |A| \cdot |B| $$
  with $A = f(y)-f(x)$ and $B = f(y)-f(z)$, which is easily verified.
\end{proof}

An immediate consequence of the lemma is the following general lower $\theta$-curvature bound for infinite dimension.

\begin{theorem} \label{thm:univgenlowbondinfdim}
  Let $(X,Q,\pi)$ be an irreducible, reversible finite Markov chain, $Q_{\min} := \min_{x \sim y} Q(x,y)$, and $\theta$ an arbitrary mean. Then $(X,Q,\pi)$ satisfies $CD_\theta(-K,\infty)$ with
  $$ K = \frac{1}{2} + \frac{2}{Q_{\min}}. $$
\end{theorem}

\begin{proof}
  It follows from Definition \ref{defn: CDtheta} and \eqref{eq:BrhoArhoLaprel} that $(X,Q,\pi)$ satisfies $CD_\theta(-K,\infty)$ if we have for all $\rho \in \mathcal{P}_*(X)$ and all $f \in \mathbb{R}^X$,
  $$ \frac{1}{2} \langle P\rho,\Gamma_\rho f \rangle_{\pi} + \left( K - \frac{1}{2} \right) \mathcal{A}_\rho(f) \ge \mathcal{A}_\rho(f,\Delta f). $$
  Choosing $C=1$ in Lemma \ref{lem:univlowbduseful} yields for all $\rho \in \mathcal{P}_*(X)$ and all $f \in \mathbb{R}^X$,
  $$ \frac{1}{2} \langle P \rho,\Gamma_\rho f \rangle_\pi + \frac{2}{Q_{\min}} \mathcal{A}_\rho(f) \ge \mathcal{A}_\rho(f,\Delta,f), $$
  which proves the theorem.
\end{proof}

\subsection{Lower entropic curvature bound for finite dimension}

In this subsection we prove the following general lower entropic curvature bound for finite dimension.

\begin{theorem} \label{thm:univgenlowbondfindim}
Let $(X,Q,\pi)$ be an irreducible, reversible finite Markov chain, $Q_{\min}:= \min_{x \sim y} Q(x,y)$, and $N \in (0,\infty]$. Then
$(X,Q,\pi)$ satisfies $CD_{ent}(-K,N)$ with
$$ K = \frac{1}{2} + \frac{4}{Q_{\min}}\left( 1 + \frac{4}{N^2} \right). $$
\end{theorem}

Note that this result is slightly weaker than the result in Theorem \ref{thm:univgenlowbondinfdim} in the case $N = \infty$.
For the proof of Theorem \ref{thm:univgenlowbondfindim}, we follow similar arguments as in the previous section. The additional ingredient in the proof are the inequalities of Lemma \ref{lem:ineqforgenlowbd}.

\begin{proof}[Proof of Theorem \ref{thm:univgenlowbondfindim}]
  For any constant $K \ge \frac{1}{2}$, $(X,Q,\pi)$ satisfies $CD_{ent}(-K,N)$ if we have 
  \begin{equation} \label{eq:lowcurvbdcondfindim} 
  \frac{1}{2} \langle P\rho,\Gamma_\rho f \rangle_{\pi} + \left( K - \frac{1}{2} \right) \mathcal{A}_\rho(f) \ge \mathcal{A}_\rho(f,\Delta f) + \frac{1}{N} \langle \rho,(\Delta f)^2 \rangle_\pi 
  \end{equation}
  for all $\rho \in \mathcal{P}_*(X)$ and all $f \in \mathbb{R}^X$. We have
  \begin{align*}
    \langle \rho, (\Delta f)^2 \rangle_\pi &= \sum_y \rho_y \left( \sum_x Q(y,x)(f(x)-f(y)) \right) \left( \sum_z Q(y,z)(f(z)-f(y)\right) \pi(x) \\
    &= \sum_{x,y,z} q_{xyz} \rho_y (f(y)-f(x))(f(y)-f(z)),
  \end{align*}
  with $q_{xyz}$ given in \eqref{eq:qxyz}. Using this, together with \eqref{eq:PrhoGammarhofnorm}, \eqref{eq:Arhoqxyz} and \eqref{eq:ArhofDfqxyz}, we conclude that \eqref{eq:lowcurvbdcondfindim} holds if we have the following termwise inequalities
  for all $x \sim y \sim z$ and all $A = f(y)-f(x), B = f(y) - f(z) \in \mathbb{R}$:
  \begin{multline*}
  Q_{\min} (A^2 b_0(\rho_z,\rho_y,\rho_x) + B^2 b_0(\rho_x,\rho_y,\rho_z)) + (2K-1) (A^2 \widehat \rho_{xy} + B^2 \widehat \rho_{zy} ) \\ \ge 4AB\left( -(\widehat \rho_{xy} + \widehat \rho_{zy}) + \frac{2 \rho_y}{N} \right).
  \end{multline*}
  Rearrangement yields
  \begin{multline*}
      \left( Q_{\min} b_0(\rho_z,\rho_y,\rho_x) +(2K-1) \widehat \rho_{xy} \right) A^2 + \left( Q_{\min} b_0(\rho_x,\rho_y,\rho_z) +(2K-1) \widehat \rho_{zy} \right) B^2 \\ \ge \left( -2(\widehat \rho_{xy} + \widehat \rho_{zy}) + \frac{4 \rho_y}{N} \right) 2AB.
  \end{multline*}
  This is inequality is of the form
  $$ C_1 A^2 + C_2 B^2 \ge 2 C_3 A B $$
  with $C_1,C_2 \ge 0$, since we assumed $K \ge \frac{1}{2}$. This inequality holds if $C_1 C_2 \ge C_3^2$, so $(X,Q,\pi)$ satisfies $CD_{ent}(-K,N)$ if we have
  \begin{equation*}
  \left( Q_{\min} b_0(\rho_z,\rho_y,\rho_x) +(2K-1) \widehat \rho_{xy} \right) \left( Q_{\min} b_0(\rho_x,\rho_y,\rho_z) +(2K-1) \widehat \rho_{zy} \right)
  \ge \left( -2(\widehat \rho_{xy} + \widehat \rho_{zy}) + \frac{4 \rho_y}{N} \right)^2.
  \end{equation*}
  As $K \ge \frac{1}{2}$, we cross-multiply and conclude that this inequality is satisfied as soon as we have
  \begin{equation*}
  (2K-1)Q_{\min}(b_0(\rho_z,\rho_y,\rho_x)\widehat \rho_{zy} + b_0(\rho_x,\rho_y,\rho_z) \widehat \rho_{xy})
  \ge \left( -2(\widehat \rho_{xy} + \widehat \rho_{zy}) + \frac{4 \rho_y}{N} \right)^2.
  \end{equation*}
  Since $(a-b)^2 \le a^2 + b^2$ for $a,b \ge 0$, in inequality holds, in turn, if we have
  \begin{equation*} 
  (2K-1)Q_{\min}(b_0(\rho_z,\rho_y,\rho_x)\widehat \rho_{zy} + b_0(\rho_x,\rho_y,\rho_z) \widehat \rho_{xy}) 
  \ge 8 ((\widehat \rho_{xy})^2 + (\widehat \rho_{zy})^2) + \frac{16 \rho_y^2}{N^2}.
  \end{equation*}
  It follows from \eqref{eq:b0est1} and \eqref{eq:b0est2} in Lemma \ref{lem:ineqforgenlowbd} that
  \begin{equation*}
  \left( 8 + \frac{32}{N^2} \right) (b_0(\rho_z,\rho_y,\rho_x)\widehat \rho_{zy} + b_0(\rho_x,\rho_y,\rho_z) \widehat \rho_{xy}) 
  \ge 8 ((\widehat \rho_{xy})^2 + (\widehat \rho_{zy})^2) + \frac{16 \rho_y^2}{N^2},
  \end{equation*}
  which means that $(X,Q,\pi)$ satisfies $CD_{ent}(-K,N)$ if we have
  $$ (2K-1) Q_{\min} \ge 8 + \frac{32}{N^2}, $$
  that is
  $$ K = \frac{1}{2} + \frac{4}{Q_{\min}}\left( 1 + \frac{4}{N^2} \right). $$
\end{proof}

\subsection{Perturbation results}
\label{subsec:perturb}

In this subsection we derive curvature estimates for pairs of ``closeby'' Markov chains with associated ``closeby'' means. We also introduce the notion of a \emph{$\delta$-curvature} for Markov chains with associated means, which is based on a restriction of the probability densities $\rho \in \mathcal{P}_*(X)$, which have a $\delta$-bound on their logarithmic gradients.

\begin{lemma} \label{lem:Bdiffperturb}
Let $(X,Q,\pi)$ and $(X,\widetilde Q,\widetilde \pi)$ be two irreducible, reversible finite Markov chains with the same underlying combinatorial graph and corresponding means $\theta$ and $\widetilde \theta$.
Assume we have for some $\rho \in \mathcal{P}_*(X)$ and some non-constant $f \in \mathbb{R}^X$:
\begin{align}
    \left\vert \log \widetilde{\mathcal{A}}_\rho(f) - \log \mathcal{A}_\rho(f) \right\vert &\le  \epsilon, \label{eq:Aest} \\
    \left\vert \log \langle \widetilde P \rho, \widetilde \Gamma_\rho f \rangle_{\widetilde \pi} - \log \langle P \rho, \Gamma_\rho f \rangle_{\pi}\right\vert &\le \epsilon, \label{eq:Pest} \\
    \left\vert \log [\widetilde{\mathcal{A}}_\rho]^\pm(f,\widetilde \Delta f) - \log [\mathcal{A}_\rho]^\pm(f,\Delta f) \right\vert &\le \epsilon, \label{eq:Apmest}
    \end{align}
where we define $\log 0 -\log 0 = 0$.    
Then we have
$$ \widetilde{\mathcal{B}}_\rho(f) - \mathcal{B}_\rho(f) \ge - \sinh(\epsilon) \left( 3 [\mathcal{B}_\rho(f)]^+ + \left( 2+ \frac{16}{Q_{\rm{min}}} \right) \mathcal{A}_\rho(f) \right) $$
with
$$ Q_{\rm{min}} = \inf_{x \sim y} \left\{ Q(x,y), \widetilde Q(x,y) \right\}. $$
\end{lemma}

\begin{proof} Note that expressions $\mathcal{A}_\rho(f), \langle P \rho,\Gamma_\rho f \rangle_\pi, [\mathcal{A}_\rho]^{\pm}(f,\Delta f)$ and their tilde-counterparts are all non-negative.
It follows from \eqref{eq:ArhoGamma} and \eqref{eq:BrhoArhoLaprel} that 
\begin{align}
    \widetilde{\mathcal{B}}_\rho(f) &= \frac{1}{2} \langle \widetilde P \rho, \widetilde \Gamma_\rho f \rangle_{\widetilde \pi} - \frac{1}{2} \widetilde{\mathcal{A}}_\rho(f) + [\widetilde{\mathcal{A}}_\rho]^-(f,\widetilde \Delta f) - [\widetilde{\mathcal{A}}_\rho]^+(f,\widetilde \Delta f) \nonumber \\
    &\ge e^{-\epsilon} \left( \frac{1}{2} \langle P \rho,\Gamma_\rho f\rangle_\pi + [\mathcal{A}_\rho]^-(f,\Delta f) \right) - e^\epsilon \left( \frac{1}{2} \mathcal{A}_\rho(f) + [\mathcal{A}_\rho]^+(f,\Delta f)\right) \nonumber\\
    &\ge \frac{e^{-\epsilon}}{2} \langle P \rho,\Gamma_\rho f \rangle_\pi - \mathcal{A}_\rho(f,\Delta f) - \frac{e^\epsilon}{2} \mathcal{A}_\rho(f) + (1-e^\epsilon)|\mathcal{A}_\rho|(f,\Delta f). \label{eq:Brhopsiest}
\end{align}
    Choosing $C = 2$ in Lemma \ref{lem:univlowbduseful} yields 
$$ \frac{1}{2} \langle P \rho,\Gamma_\rho f \rangle_\pi + \frac{8}{Q_{\rm{min}}} \mathcal{A}_\rho(f) - 2 \vert \mathcal{A}_\rho \vert(f,\Delta f) \ge 0. $$
and subtracting $L$-times this expression from \eqref{eq:Brhopsiest}, we obtain
    \begin{align*}
    \widetilde{\mathcal{B}}_\rho(f) &\ge\frac{1}{2}\left( e^{-\epsilon}-L\right) \langle P \rho,\Gamma_\rho f \rangle_\pi - \left( \frac{e^\epsilon}{2} + \frac{8L}{Q_{\rm{min}}} \right) \mathcal{A}_\rho(f) \\ &\quad- \mathcal{A}_\rho(f,\Delta f) - \left( e^{\epsilon}-1-2L \right) |\mathcal{A}_\rho|(f,\Delta f).
    \end{align*}
    By choosing $L=e^\epsilon - e^{-\epsilon}$, the coefficients of $\langle P \rho,\Gamma_\rho f \rangle_\pi$ and $\mathcal{A}_\rho(f,\Delta f)$ match up and since $e^{\epsilon}-1-2L \le 0$, we end up with
    \begin{align*} 
    \widetilde{\mathcal{B}}_\rho(\psi) &\ge \left( 2e^{-\epsilon} - e^\epsilon \right) \left( \frac{1}{2} \langle P \rho,\Gamma_\rho f \rangle_\pi - \mathcal{A}_\rho(f,\Delta f) \right) - \left( \frac{e^\epsilon}{2} + \frac{8\left( e^\epsilon - e^{-\epsilon}\right)}{Q_{\rm{min}}} \right) \mathcal{A}_\rho(f) \nonumber \\
    &= (2e^{-\epsilon}-e^\epsilon) \mathcal{B}_\rho(f) -\left( e^\epsilon-e^{-\epsilon} \right) \left( 1+ \frac{8}{Q_{\rm{min}}}\right) \mathcal{A}_\rho(f) \\
    &\ge \mathcal{B}_\rho(f) - \sinh(\epsilon)  \left( 3 [\mathcal{B}_\rho(f)]^+ + 2\left( 1 + \frac{8}{Q_{\rm{min}}} \right) \mathcal{A}_\rho(f) \right).
    \end{align*}
\end{proof}

\begin{definition} \label{def:Kdelta}
Let $(X,Q,\pi)$ be an irreducible, reversible finite Markov chain, $\theta$ be a mean, and $\delta > 0$. We define 
$$ 
K_\delta := \inf_{\text{$f$ not constant}}
\inf \left\{ 
\frac{\mathcal{B}_\rho(f)}{\mathcal{A}_\rho(f)}: \text{$\rho \in \mathcal{P}_*(X)$ with $| \nabla \log \rho | \le \delta$ on ${\rm{supp}} \nabla f$}  
\right\} 
$$
\end{definition}

\begin{remark}
  The requirement $|\nabla \log \rho| < \delta$ in the above definition has been proved useful in the discrete setting. It has been called the \emph{regularization trick} in \cite{STY-23}, and has been used there to compare the log-Sobolev constant with its modified counterpart (see Theorem 1 in \cite{STY-23}). Moreover, we have
  \begin{equation} \label{eq:Kdeltabounds}
     - \frac{1}{2} - \frac{2}{Q_{\min}} \le K_\theta(X,Q,\pi) \le K_\delta \le 2,     
  \end{equation}
  where the lower bound follows from 
  Theorem \ref{thm:univgenlowbondinfdim}, and the upper bound follows by choosing $\rho$ constant and $\psi_1$ the eigenfunction of $-\Delta$ to the first positive eigenvalue $\lambda \le 2$ (note that we have $\langle \rho, \Gamma_{2,\rho} \psi_1 \rangle_\pi = \lambda \langle \rho, \Gamma_\rho(\psi_1) \rangle_\pi$).
\end{remark}

\begin{theorem} \label{thm:perturb1}
Let $\epsilon \in (0,\frac{\rm{arcsinh(1/3)}}{4})$ and $\delta \in (0,\infty)$.
Let $(X,Q,\pi)$ and $(X, \widetilde Q,\widetilde \pi)$ be two irreducible, reversible finite Markov chains with the same underlying combinatorial graphs satisfying
$$ \left\Vert \log \frac{\widetilde Q}{Q} \right\Vert_\infty \le \epsilon \quad \text{and} \quad \left\Vert \log \frac{\widetilde \pi}{\pi} \right\Vert_\infty \le \epsilon, $$
where we define $\log(0/0) = 0$.
Assume we also have means $\theta, \widetilde \theta$ associated to $(X,Q,\pi)$ and $(X,\widetilde Q,\widetilde \pi)$, respectively, satisfying
\begin{equation} \label{eq:theta1-assum} 
\left| \log \frac{\partial_1 \widetilde \theta(s,1)}{ \partial_1 \theta(s,1)} \right| \le \epsilon  \quad \text{whenever} \quad s \in \left[ e^{-\delta},e^\delta \right]. 
\end{equation}
Then we have 
$$ \left| \widetilde K_\delta - K_\delta \right| \le \sinh(4\epsilon) \left( 17+ \frac{36}{Q_{\rm{min}}}\right) $$
where 
$$ Q_{\rm{min}} = \inf_{x \sim y}
\left\{ Q(x,y), \widetilde Q(x,y) \right\}. 
$$
\end{theorem}

\begin{proof}
  Let $f \in \mathbb{R}^X$ be non-constant and assume that $\rho \in \mathcal{P}_*(X)$ satisfies
  $$ \left| \nabla \log \rho(x,y) \right| \le \delta \quad \text{for all $x \sim y$ with $f(y) \neq f(x)$.} $$
  Then we have for $x \sim y$ with $(x,y) \in {\rm{supp}} (\nabla f)$ that $s = \frac{\rho(y)}{\rho(x)} \in [e^{-\delta},e^\delta]$  and therefore
  $$ \partial_1 \theta(\rho_x,\rho_y) = \partial_1 \theta(1,s) $$
  and similarly for $\partial_1 \widetilde \theta$. By assumption \eqref{eq:theta1-assum} we conclude
  $$ \left| \log \frac{\partial_1 \widetilde \theta(\rho_x,\rho_y)}{ \partial_1 \theta(\rho_x,\rho_y)} \right| \le \epsilon  \quad \text{for all $x \sim y$ with $(x,y) \in {\rm{supp}} (\nabla f)$}. $$
  Moreover, we have for $q_{xyz} = Q(y,x)Q(y,z)\pi(y)$
  and $\widetilde q_{xyz} = \widetilde Q(y,x)\widetilde Q(y,z)\widetilde \pi(y)$
  $$ \left\vert \log \frac{\widetilde q_{xyz}}{q_{xyz}}  \right\vert \le 3 \epsilon. $$
  Using the expressions \eqref{eq:Arhoqxyz} and \eqref{eq:Arhopm} for $\mathcal{A}_\rho(f)$ and $[\mathcal{A}_\rho]^\pm(f,\Delta f)$ (and their tilde-analogues), respectively, as well as
  $$ \langle P \rho, \Gamma_\rho f \rangle_\pi = \sum_{x,y,z} q_{xyz}\, \rho_z\, \partial_1 \theta(\rho_y,\rho_x)\, (f(y)-f(x))^2
  $$ 
  (and its tilde-analogue), their positivity properties 
  and the relation \eqref{eq:rhoasderiv}: 
  $$ \widehat\rho_{uv} = \rho(u)\partial_1 \theta(\rho_u,\rho_v) + \rho(v) \partial_1 \theta(\rho_v,\rho_u), $$
  we see that the three main assumptions \eqref{eq:Aest}, \eqref{eq:Pest}  and \eqref{eq:Apmest} in Lemma \ref{lem:Bdiffperturb} are satisfied with $4 \epsilon$ instead of $\epsilon$. In particular, we have
  \begin{equation} \label{eq:AtildeAexpr}
  \frac{\left| \widetilde{\mathcal{A}}_\rho(f) - \mathcal{A}_\rho(f) \right|}{\widetilde{\mathcal{A}}_\rho(f)} = \left| 1 - \frac{\mathcal{A}_\rho(f)}{\widetilde{\mathcal{A}}_\rho(f)} \right| \le e^{4 \epsilon} - 1 \le 2 \sinh(4 \epsilon). 
  \end{equation}
  Note that 
  \begin{equation} \label{eq:KdeltaB} 
  K_\delta \mathcal{A}_\rho(f) \le \mathcal{B}_\delta(f) \quad \text{and} \quad  K_\delta \le 2, 
  \end{equation}
  where the upper bound on $K_\delta$ 
  follows by choosing $\rho$ constant and $\psi_1$ the eigenfunction of $-\Delta$ to the first positive eigenvalue $\lambda \le 2$ (note that $\psi_1$ is not constant and we have $\langle \rho, \Gamma_{2,\rho} \psi_1 \rangle_\pi = \lambda \langle \rho, \Gamma_\rho(\psi_1) \rangle_\pi$).
  The restriction on $\epsilon > 0$ yields $e^{4 \epsilon}\le 2$, and we obtain, together with Lemma \ref{lem:Bdiffperturb}, \eqref{eq:Kdeltabounds}, \eqref{eq:AtildeAexpr} and \eqref{eq:KdeltaB} in the case $\mathcal{B}_\rho(f) \ge 0$,
  \begin{align*}
  \widetilde{\mathcal{B}}_\rho(f) &\ge \underbrace{\left( 1-3\sinh(4\epsilon) \right)}_{>0} \mathcal{B}_\rho(f) - \sinh(4\epsilon)\left( 2 + \frac{16}{Q_{\rm{min}}} \right) \mathcal{A}_\rho(f) \\ 
  &\ge \left( K_\delta - \sinh(4\epsilon)\left( 3 K_\delta + 2 + \frac{16}{Q_{\rm{min}}} \right)\right) \mathcal{A}_\rho(f) \\
  & \ge K_\delta \widetilde {\mathcal{A}}_\rho(f) - |K_\delta|\cdot \left| \widetilde{\mathcal{A}}_\rho(f)-\mathcal{A}_\rho(f) \right| 
  - e^{4 \epsilon} \sinh(4 \epsilon) \left( 8 + \frac{16}{Q_{\min}} \right) \widetilde{\mathcal{A}}_\rho(f) \\
  &\ge \left( K_\delta - \left( \frac{1}{2} + \frac{2}{Q_{\min}}  \right) 2 \sinh(4\epsilon) - 2 \sinh(4\epsilon) \left( 8 + \frac{16}{Q_{\min}} \right) \right) \widetilde{\mathcal{A}}_\rho(f) \\
  &= \left( K_\delta - \sinh(4\epsilon) \left( 17 + \frac{36}{Q_{\min}}\right) \right) \widetilde{\mathcal{A}}_\rho(f).
  \end{align*}
  It is easy to see that the case $\mathcal{B}_\rho(f) < 0$ leads to the even better lower estimate
  $$ \widetilde{\mathcal{B}}_\rho(f) \ge \left( K_\delta - \sinh(4\epsilon) \left( 5 + \frac{36}{Q_{\min}}\right)\right) \widetilde{\mathcal{A}}_\rho(f). $$
  Hence we obtain
  $$ \widetilde K_\delta - K_\delta \ge - \sinh(4\epsilon) \left( 17 + \frac{36}{Q_{\rm{min}}} \right). $$
  Since this inequality holds also by interchanging $K_\delta$ and $\widetilde K_\delta$, the claim of the theorem follows.
\end{proof}

The final theorem in this section states a perturbation result in the setting of a connected finite weighted graph $G=(V,E)$ with vertex set $V$, edge set $E$, vertex measure $m: V \to (0,\infty)$ and symmetric edge weights $w: V \times V \to [0,\infty)$. We refer to such a graph as the $4$-tuple $(V,E,m,w)$. Assuming 
\begin{equation} \label{eq:mcond}
m(x) \ge \sum_y w_{xy},
\end{equation}
such a $4$-tuple represents a reversible Markov chain $(X,Q,\pi)$ with $X = V$, $Q(x,y) = \frac{w_{xy}}{m(x)}$ for $x \neq y$ and $Q(x,x) = 1 - \sum_y Q(x,y) \ge 0$.

\begin{theorem} \label{thm:perturb2}
Let $G=(V,E)$ be a connected finite graph with two choices of vertex measures $m, \widetilde m: B \to (0,1]$ and symmetric edge weights $w, \widetilde w: V \times V \to [0,1]$ satisfying \eqref{eq:mcond}. Let $(X,Q,\pi)$ and $(X,\widetilde Q,\widetilde \pi)$ be the Markov chains associated to $(V,E,m,w)$ and $(V,E,\widetilde m,\widetilde w)$, respectively.  
Assume there exists $\epsilon > 0$ such that
$$ \sup_{x \sim y} \left\vert \log \frac{\widetilde w_{xy}}{w_{xy}} \right \vert \le \epsilon \quad \text{and} \quad \sup_{x \in X} \left\vert \log \frac{\widetilde m(x)}{m(x)} \right\vert \le \epsilon. $$
Let $\theta$ be a mean. Then we have 
$$ | K_\theta(X,Q,\pi)) - K_\theta(\widetilde X,\widetilde Q,\widetilde \pi) | \le 27 \epsilon \left( 1 + \frac{8}{Q_{\rm{min}}} \right), $$
where 
$$Q_{\rm{min}} := \inf_{x \sim y} \min\{ Q(x,y),\widetilde Q(x,y)\}.$$
\end{theorem}

\begin{proof}
    We first assume that $m(x) \ge 2 \sum_{y \sim x} w_{xy}$ and $\widetilde m(x) \ge 2 \sum_{y \sim x} \widetilde w_{xy}$ (this can always be achieved by rescaling the vertex measures $m$ and $\widetilde m$ by the same multiplicative factor $\ge 1$). 
    Note that we have
    $$ \sup_{x \sim y} \left\vert \log \frac{\widetilde Q_{xy}}{Q_{xy}} \right \vert \le 2\epsilon. $$
    For the estimate on the diagonal, we define $A = \sum_{y \sim x} Q(x,y) \le 1/2$ and $\widetilde A = \sum_{y \sim x} \widetilde Q(x,y) \le 1/2$ and use $A \ge e^{-2\epsilon} \widetilde A$ to obtain
    $$
    \frac{1-A}{1-\widetilde A} \le \frac{1-e^{-2\epsilon}\widetilde A}{1-\widetilde A} = 1 + \frac{\widetilde A - e^{-2\epsilon}\widetilde A}{1-\widetilde A} \le 1 + 2(1-e^{-2\epsilon})\widetilde A \le 2 - e^{-2\epsilon}.
    $$
    This implies
    $$ \log\left( \frac{Q_{xx}}{\widetilde Q_{xx}} \right) = \log\left( \frac{1-A}{1-\widetilde A} \right) \le \log\left( 2-e^{-2\epsilon} \right) \le 2 \epsilon, $$
    which shows that we also have
    $$ \sup_{x \in X} \left\vert \log \frac{\widetilde Q_{xx}}{Q_{xx}}\right\vert \le 2\epsilon. $$
    Note that the steady states are obtained from the vertex measures via normalization:
    $$ \pi(x) = \frac{m(x)}{\sum_y m(y)} \quad \text{and} \quad \widetilde \pi(x) = \frac{\widetilde m(x)}{\sum_y \widetilde m(y)}. $$
    Therefore, we have both
    $$ \left\Vert \log \frac{\widetilde Q}{Q} \right\Vert_\infty \le 2 \epsilon \quad \text{and} \quad \left\Vert \log \frac{\widetilde \pi}{\pi} \right\Vert_\infty \le 2 \epsilon, $$
    where we define $\log(0/0) = 0$. 
    This implies for any $x,y,z$ that
    $$ \left\vert \log \frac{\widetilde q_{xyz}}{q_{xyz}} \right\vert \le 3 \epsilon. $$
    Similarly as before, we can therefore assume that \eqref{eq:Aest}, \eqref{eq:Pest} and \eqref{eq:Apmest} hold with  $3\epsilon$ instead of $\epsilon$.
    Let $f \in \mathbb{R}^X$, $\rho \in \mathcal{P}_*$, and $\widetilde K = K_\theta(\widetilde X,\widetilde Q,\widetilde \pi)$.
    Then 
    \begin{align*}
    \frac{1}{2} e^{3\epsilon} \langle P \rho, \Gamma_\rho f \rangle_\pi &\ge \frac{1}{2} \langle \widetilde P \rho, \widetilde \Gamma_\rho f \rangle_{\widetilde \pi} \\ 
    &\ge \left( \widetilde K + \frac{1}{2}\right) \widetilde{\mathcal{A}}_\rho(f) + \widetilde{\mathcal{A}}_\rho(f,\widetilde \Delta f) \\
    &\ge e^{-3\epsilon}\left( \left[ \widetilde K + \frac{1}{2}\right]^+ \mathcal{A}_\rho(f) + [\mathcal{A}_\rho]^+(f, \Delta f) \right) \\ &\quad- e^{3 \epsilon} \left( \left[ \widetilde K + \frac{1}{2}\right]^- \mathcal{A}_\rho(f) + [\mathcal{A}_\rho]^-(f,\Delta f) \right)\\
    &\ge \left( \left( \widetilde K + \frac{1}{2} \right) - \left\vert \widetilde K + \frac{1}{2} \right\vert \left( e^{3\epsilon}-1 \right) \right) \mathcal{A}_\rho(f) \\
    &\quad + \mathcal{A}_\rho(f,\Delta f) - (e^{3 \epsilon}-1) \vert\mathcal{A}_\rho\vert(f,\Delta f). 
    \end{align*}
    Adding this inequality to $t > 0$ times the inequality \eqref{eq:infcurvlowbd} with the choice $C=2$, we obtain
    \begin{multline*}
    \frac{1}{2} \left( t + e^{3 \epsilon} \right) \langle P \rho, \Gamma_\rho f \rangle_\pi \ge \left( \left( \widetilde K + \frac{1}{2} \right) - \left\vert \widetilde K + \frac{1}{2} \right\vert \left( e^{3\epsilon}-1 \right)
    - \frac{8t}{Q_{\rm{min}}} \right) \mathcal{A}_\rho(f) \\ + \left( 2 t + 1 - e^{3 \epsilon}\right) |\mathcal{A}_\rho|(f,\Delta f) + \mathcal{A}_\rho(f,\Delta f).
    \end{multline*}
    We choose $t = 2e^{3\epsilon}-2$. This choice minimizes $8t$ while satisfying the conditions
    \begin{align*}
    2t+1-e^{3\epsilon} &\ge 0, \\
    t + e^{3 \epsilon} &\le 2t + 1 - e^{3 \epsilon} +1.
    \end{align*}
    Plugging this into the inequality and dividing by $3e^{3 \epsilon}-2$, we obtain
    \begin{equation*}
\frac{1}{2} \langle P \rho, \Gamma_\rho f \rangle_\pi \ge \frac{1}{3e^{3\epsilon} - 2} \left( \left( \widetilde K + \frac{1}{2} \right) - \left\vert \widetilde K + \frac{1}{2} \right\vert \left( e^{3\epsilon}-1 \right)
    - \frac{16(e^{3\epsilon}-1)}{Q_{\rm{min}}} \right) \mathcal{A}_\rho(f) + \mathcal{A}_\rho(f,\Delta f).
    \end{equation*}
    Hence, setting $K = K_\theta(X,Q,\pi)$,
    \begin{align*} 
    K + \frac{1}{2} &\ge \frac{1}{3e^{3\epsilon} - 2} \left( \left( \widetilde K + \frac{1}{2} \right) - \left\vert \widetilde K + \frac{1}{2} \right\vert \left( e^{3\epsilon}-1 \right)
    - \frac{16(e^{3\epsilon}-1)}{Q_{\rm{min}}} \right) \\
    &= O(\epsilon^2) + (1-9\epsilon) \left( \left( \widetilde K + \frac{1}{2} \right) - 3 \epsilon \left\vert \widetilde K + \frac{1}{2} \right\vert  - \frac{48 \epsilon}{Q_{\rm{min}}} \right) \\
    &= O(\epsilon^2) + \left( \widetilde K + \frac{1}{2} \right)
    - 3 \epsilon \left( 3 \left(\widetilde K + \frac{1}{2} \right) + \left\vert \widetilde K + \frac{1}{2} \right\vert + \frac{16}{Q_{\rm{min}}} \right). 
    \end{align*}
    Using 
    $$ - \frac{1}{2} - \frac{2}{Q_{\rm{min}}} \le \widetilde K \le 1, $$
    where the lower bound follows from Theorem \ref{thm:univgenlowbondfindim}, and the upper bound follows from the fact that $\sum_y Q(x,y) \le 1/2$ and by choosing $\rho$ constant and $\psi_1$ the eigenfunction of $-\Delta$ to the first positive eigenvalue $\lambda \le 2$,
    we obtain
    $$ K \ge  \widetilde K - 27 \epsilon \left( \frac{1}{2} + \frac{2}{Q_{\rm{min}}} \right) + O(\epsilon^2). $$
    Reversing the roles of $K$ and $\widetilde K$ yields
    $$ | K - \widetilde K | \le 27 \epsilon \left( \frac{1}{2} + \frac{2}{Q_{\rm{min}}} \right) + O(\epsilon^2), $$
    where the $O(\epsilon^2)$-term depends only on $Q_{\rm{min}}$. This term disappears due to local Lipschitz continuity of the curvature which we have just proven.

    Relaxing the initial condition to $m(x) \ge \sum_{y \sim x} w_{xy}$ (and similarly for $\widetilde m(x)$), rescaling shows that we end up with the bound
    $$ | K - \widetilde K | \le 27 \epsilon \left( 1 + \frac{8}{Q_{\rm{min}}} \right). $$
\end{proof}

\section{Examples}
\label{sec:examples}

In this section, we present examples of finite Markov chains $(X,Q,\pi)$ with very different behaviours with regards to their entropic curvatures and Bakry-\'Emery and Ollivier Ricci curvatures. To simplify the presentation, we drop the requirement of $\pi$ and $\rho$ being normalised. Note, however, that we require $\sum_{y} Q(x,y) =1$ for our transition probabilities. We also present comparisons with Ollivier's sectional curvature \cite{Oll-09} (see also \cite{Mu-23}). Let us briefly recall Ollivier curvature notions. 

\begin{definition} \label{def:OllRicsec}
Let $(X,Q,\pi)$ finite Markov chain with underlying combinatorial graph $G = (V,E)$ and combinatorial distance $d: V \times V \to \mathbb{N} \cup \{0\}$. 

The \emph{Ollivier Ricci curvature} of $(X,Q,\pi)$ (see \cite{Oll-09}) is defined as follows on the edges $x \sim y$ of $G$:
$$ K_{ORC}(x,y) = 1 - W_1(Q(x,\cdot),Q(y,\cdot)), $$
where $W_1$ is the $L^1$-Wasserstein distance for probability measures on $X$, defined as
$$W_1(\mu,\nu):=\inf_{\gamma:X\times X\to [0,1]}\left\{ \sum_{u,v\in V} d(u,v)\gamma(u,v)\, \middle| \, \sum_{v}\gamma(u,v)=\mu(u) \, \text{ and } \, \sum_{u}\gamma(u,v)=\nu(v) \right\}.$$

The \emph{Ollivier sectional curvature} of $(X,Q,\pi)$  (see \cite[Problem P]{Oll-10}) is defined as follows on the edges $x \sim y$ of $G$:
$$ K_{SEC}(x,y) = 1 - W_\infty(Q(x,\cdot),Q(y,\cdot)), $$
where $W_\infty$ is the $L^\infty$-Wasserstein distance for probability measures on $X$, defined by
$$ W_\infty(\mu,\nu) = \inf_{\pi_{\mu,\nu}} \sup_{\substack{x,y: \\ \pi_{\mu,\nu}(x,y) > 0}} 
d(x,y), $$
where the infimum runs over all transport plans $\pi_{\mu,\nu}$ from the probability measure $\mu$ to the probability measure $\nu$. It is not difficult to see that, assuming laziness $\ge 1/2$, this sectional curvature only assumes the values $-1,0$. We then say that Ollivier sectional curvature of $(X,Q,\pi)$ is negative if $K_{SEC}(x,y)=-1$ on some edge $x\sim y$, and nonnegative if $K_{SEC}=0$ for all edges.
\end{definition}

Ollivier Ricci curvature $K_{ORC}(x,y)$ can also be computed alternatively via the following formula from \cite[Theorem 2.1]{MW-19}:
\begin{equation} \label{eq:OR_formula}
    K_{ORC}(x,y)=\inf_{\substack{f\in Lip(1) \\ f(y)-f(x)=1}} (\Delta f(x)-\Delta f(y)).
\end{equation}

\begin{remark}
Note that we always have
$$ K_{ORC}(x,y) \ge K_{SEC}(x,y) $$
since $W_1 \le W_\infty$.
\end{remark}

To make the notation easier to read and cohesive with other curvature notions, in this section we will write $K_{BE}$ for the Bakry-\'Emery curvature for $\infty$-dimension.

A Markov chain $(X,Q,\pi)$ with $X=\{0,1,\ldots,n-1\}$ is a \emph{birth-death process} if
\begin{equation*}
    Q(i,j)=0 \quad \text{whenever} \quad |j-i|>1
\end{equation*}
Equivalently, its associated weighted graph is supported on an $n$-point path.
For abbreviation, we write $Q^\pm(r) = Q(r,r\pm 1)$ (and by convention, $Q^\pm(r) =0$ in case $r \pm 1$ is no longer a vertex).

We recall the formulae for computing Ollivier Ricci curvature and Bakry-\'Emery curvature of birth-death processes, which were given in \cite[Theorem 2.10]{MW-19} and \cite[Section 2.1]{HM-23}, respectively. We remark further that these formulae are also applicable for weighted cycle graphs of length at least $6$.

\begin{theorem}[{\cite[Theorem 2.10]{MW-19}}] \label{thm:OL_line_formula}
Let $(X,Q,\pi)$ be a birth-death chain. Suppose $Q(x,x) \ge \frac{1}{2}$ for all vertices $x\in X$.
Then Ollivier Ricci curvature on the edge $(n,n+1)$ is given by
$$ K_{ORC}(n,n+1) = Q^+(n) + Q^-(n+1) - Q^+(n+1) - Q^-(n), $$
\end{theorem}

\begin{theorem}[{\cite[Section 2.1]{HM-23}}] \label{thm:BE_line_formula}
Let $(X,Q,\pi)$ be a birth-death chain. Let $n\in X$. Denote
\begin{align*}
    W_-(n) &:=-Q^-(n-1)+3Q^+(n-1)+Q^-(n)-Q^+(n), \text{ and} \\
    W_+(n) &:=-Q^+(n+1)+3Q^-(n+1)+Q^+(n)-Q^-(n).
\end{align*}
Then the Bakry-\'Emery curvature at the vertex $n$ can be computed as 
\begin{enumerate}
    \item If $n-1\not\in X$, then $K_{BE}(n)=W_+(n)$.
    \item If $n+1\not\in X$, then $K_{BE}(n)=W_-(n)$.
    \item Otherwise, $K_{BE}(n)=\max K$ such that
    \begin{equation} \label{eq: BE_line_formula}
    W_-(n) \ge 2K, \text{ and } \left(W_-(n)-2K\right)\left(W_+(n)-2K\right) \ge 4Q^-(n)Q^+(n).
\end{equation}
\end{enumerate}
\end{theorem}

\subsection{Weighted $3$-point graph}

{\it In the following $\alpha$-parametrized weighted three-point birth-death chain $X$ with $X=\{x,y,z\}$ where $x\sim y$ and $y\sim z$, the entropic curvature tends to $- \infty$ (as the parameter $\alpha\to\infty$), while both Bakry-\'Emery curvature and Ollivier Ricci curvature are strictly positive and bounded away from zero uniformly. Note, however, the Ollivier sectional curvature in this example is $-1$. This example was already considered as a counterexample of the Peres-Tetali conjecture about the log-Sobolov inequality for positive Ollivier curvature (see \cite{munch2023ollivier}).}


\medskip

We start with four positive constants $\pi_x,\pi_y,w_1,w_2$ whose values and relations are to be determined later. We also set a parameter $\alpha>1$ independently of aforementioned constants. Now we define a test function $f\in \mathbb{R}^X$ and measures $\rho\in\mathbb{R}_+^X$ and $\pi\in\mathbb{R}_+^X$ whose values are given according to the following tables:
\begin{center}
\begin{tabular}{|c|c c c|}
\hline
    & $x$ & $y$ & $z$ \\
\hline
    $f$& 0 & 1 & $\alpha+1$ \\
    $\rho$ & 1 & 1 & $\exp(-\alpha^2)$ \\
    $\pi$ & $\pi_x$ & $\pi_y$ & $\frac{\exp(\alpha^2)}{\alpha^2}$ \\
\hline
\end{tabular}
\end{center}
Moreover, we set $w_{xy}(=w_{yx})=w_1$ and $w_{yz}(=w_{zy})=w_2$. It is automatically implied by the detailed balance equation that $$Q(x,y)=\frac{w_1}{\pi_x}; \quad Q(y,x)=\frac{w_1}{\pi_y}; \quad Q(y,z)=\frac{w_2}{\pi_y}; \quad
Q(z,y)=\frac{\alpha^2w_2}{e^{\alpha^2}}.$$
{Note that we need $\pi_x \ge w_1$, $\pi_y \ge w_1+w_2$ and $\alpha^2w_2/e^{\alpha^2} \le 1$ to satisfy $\sum_w Q(v,w) = 1$.}

Our first aim is to show that this Markov chain has the entropic curvature tends to $-\infty$ as the parameter $\alpha$ tends to $+\infty$.

Recall that with the above choices of $\rho,f$, the entropic curvature satisfies
$$K_{ent}(X,Q,\pi) \le \frac{\mathcal{B}_\rho(f)}{\mathcal{A}_\rho(f)},$$ where (see \eqref{eq:ArhoGamma} and \eqref{eq:BrhoArhoLaprel})
\begin{align*}
    \mathcal{A}_\rho(f) &= \langle \rho, \Gamma_\rho f \rangle_\pi, \\
    \mathcal{B}_\rho(f) &=\frac{1}{2} \langle \rho, \Delta\Gamma_\rho f \rangle_\pi - \mathcal{A}_\rho(f,\Delta f).
\end{align*}

We first compute $\Gamma_\rho f$ explicitly on each vertex:
\begin{align*}
\Gamma_\rho f(x) &= (f(y)-f(x))^2 Q(x,y) \partial_1\theta(\rho_x,\rho_y)=\frac{w_1}{\pi_x}\cdot\frac{1}{2},\\
\Gamma_\rho f(y) &= (f(x)-f(y))^2 Q(y,x) \partial_1\theta(\rho_y,\rho_x)+ (f(z)-f(y))^2 Q(y,z) \partial_1\theta(\rho_y,\rho_z)\\
&=\frac{w_1}{\pi_y}\cdot\frac{1}{2}+\alpha^2 \frac{w_2}{\pi_y} \frac{\alpha^2-1+e^{-\alpha^2}}{\alpha^4} \le \frac{w_1}{2\pi_y} + \frac{w_2}{\pi_y},\\
\Gamma_\rho f(z) &= (f(y)-f(z))^2 Q(z,y) \partial_1\theta(\rho_z,\rho_y)
= \alpha^2 \cdot \frac{w_2 \alpha^2}{e^{\alpha^2}} \cdot \frac{e^{\alpha^2}-1-\alpha^2}{\alpha^4} \le w_2.
\end{align*}
It follows that
\begin{equation} \label{eq:3pt-calc-1}
    0 \le \mathcal{A}_\rho(f)= \langle \rho, \Gamma_\rho f\rangle_\pi
    = \sum_{v\in X}\rho_v\pi(v)\Gamma_\rho f(v)
    \le w_1+w_2+\frac{w_2}{\alpha^2},
\end{equation}
which is bounded as $\alpha\to\infty$.

Next we compute $\Delta\Gamma_\rho f$:
\begin{align*}
\Delta\Gamma_\rho f(x) 
&= (\Gamma_\rho f(y)-\Gamma_\rho f(x))Q(x,y) \le \left(\frac{w_1}{2\pi_y}+\frac{w_2}{\pi_y}-\frac{w_1}{2\pi_x}\right)\frac{w_1}{\pi_x}\\
\Delta\Gamma_\rho f(y), 
&= (\Gamma_\rho f(x)-\Gamma_\rho f(y))Q(y,x)+(\Gamma_\rho f(z)-\Gamma_\rho f(y))Q(y,z) \\
&\le \left(\frac{w_1}{2\pi_x}-0 \right)\frac{w_1}{\pi_y}+ (w_2-0)\frac{w_2}{\pi_y},
\\
\Delta\Gamma_\rho f(z) 
&= (\Gamma_\rho f(y)-\Gamma_\rho f(z))Q(z,y) \le \left( \frac{w_1}{2\pi_y}+\frac{w_2}{\pi_y}-0\right)\frac{w_2}{\pi(z)}.
\end{align*}
It follows that
\begin{equation} \label{eq:3pt-calc-2}
    \langle \rho, \Delta \Gamma_\rho f \rangle_\pi \le \frac{w_1^2}{2\pi_y}+\frac{w_1w_2}{\pi_y}+w_2^2+ \left(\frac{w_1w_2}{2\pi_y}+\frac{w_2^2}{\pi_y} \right)e^{-\alpha^2},
\end{equation}
which is bounded as $\alpha\to \infty$.

Next we compute $\Delta f$ on each vertex:
\begin{align*}
    \Delta f(x) &= (f(y)-f(x))Q(x,y)=\frac{w_1}{\pi_x},\\
    \Delta f(y) &= (f(x)-f(y))Q(y,x)+(f(z)-f(y))Q(y,z) = -\frac{w_1}{\pi_y}+\alpha\frac{w_2}{\pi_y},\\
    \Delta f(z) &= (f(y)-f(z))Q(z,y)= -\frac{w_2\alpha^3}{e^{\alpha^2}},
\end{align*} 
which , using \eqref{eq:Arhofg}, 
\begin{align} \label{eq:3pt-calc-3}
    \mathcal{A}_\rho (f,\Delta f)
    &=w_{xy}\hat{\rho}_{xy}(f(y)-f(x))(\Delta f(y)-\Delta f(x))+ \nonumber\\
    &\mathrel{\phantom{=}} w_{yz}\hat{\rho}_{yz}(f(z)-f(y))(\Delta f(z)-\Delta f(y)) \nonumber\\
    &= w_1 \left(\frac{\alpha w_2}{\pi_y}-\frac{w_1}{\pi_y}-\frac{w_1}{\pi_x} \right)+ \\
    &\mathrel{\phantom{=}}
    w_2\frac{1-e^{-\alpha^2}}{\alpha^2}(-\alpha)\left(-\frac{w_1}{\pi_y}+\frac{w_2 \alpha}{\pi_y}+\frac{w_2\alpha^3}{e^{\alpha^2}}\right) \nonumber,
\end{align}
which tends to infinity as $\alpha\to\infty$ (or quantitatively, $\frac{1}{\alpha}\mathcal{A}_\rho (f,\Delta f) \to \frac{w_1w_2}{\pi_y}$). From \eqref{eq:3pt-calc-1},\eqref{eq:3pt-calc-2}, \eqref{eq:3pt-calc-3}, we can conclude that the entropic curvature of this Markov chain tends to $-\infty$ as $\alpha \to \infty$.

Our second aim is to choose $\pi_x,\pi_y,w_1,w_2$ suitably such that the corresponding weighted Markov chain has positive Ollivier Ricci curvature and Bakry-\'Emery curvature, uniformly for all $\alpha > 0$. We start by choosing 
\[Q(x,y)=0.28; \quad Q(y,x)=0.1; \quad Q(y,z)=0.3.\]
Next we choose $\pi_x:=1$, which implies by the detailed balances that $\pi_y=Q(x,y)/Q(y,x)=2.8$ and $w_1=Q(x,y)\pi_x=0.28$ and $w_2=Q(y,z)\pi_y=0.84$. Moreover, in order to obtain $\pi_z=\frac{\exp(\alpha^2)}{\alpha^2}$ as in the earlier table, we require $Q(z,y)= \frac{Q(y,z)\pi_y \alpha^2}{\exp(\alpha^2)}=\frac{0.84\alpha^2}{\exp(\alpha^2)}$. For simplicity, let us call $Q(z,y)=\varepsilon$.

With large enough $\alpha$ (so that $\varepsilon<0.5$), the condition $$Q(x,x),Q(y,y),Q(z,z)\ge \frac{1}{2}$$ in Theorem \ref{thm:OL_line_formula} is satisfied. Hence by applying Theorem \ref{thm:OL_line_formula}, we have the following formulae for the Ollivier Ricci curvature:
\begin{align*}
    K_{ORC}(x,y) &=Q(x,y)+Q(y,x)-Q(y,z)=0.08,\\
    K_{ORC}(y,z) &=Q(y,z)+Q(z,y)-Q(y,x)=0.2+\varepsilon.
\end{align*}
Next, by applying Theorem \ref{thm:BE_line_formula}, we have the following formulae for the Bakry-\'Emery curvature:
\begin{align*}
    K_{BE}(x) &= \frac{1}{2}(-Q(y,z)+3Q(y,x)+Q(x,y))=0.14,\\
    K_{BE}(z) &= \frac{1}{2}(-Q(y,x)+3Q(y,z)+Q(z,y))=0.4+ \frac{1}{2}\varepsilon.
\end{align*}
For $K_{BE}(y)$, we note that 
\begin{align*}
    W_-(y) &=3Q(x,y)+Q(y,x)-Q(y,z)=0.64,\\
    W_+(y) &=3Q(z,y)+Q(y,z)-Q(y,x)=0.2+3\varepsilon,
\end{align*}and
$K_{BE}(y)=\max K$ such that $0.64\ge 2K$ and 
\[(0.64-2K)(0.2+3\varepsilon-2K)\ge 0.12.\]
One can check that $K_{BE}(y)\ge 0.004$.

Now we can see that both the Ollivier Ricci curvature on all edges and the Bakry-\'Emery curvature at all vertices are strictly positive and uniformly bounded away from zero as desired.

For Ollivier sectional curvature, note that for any transport plan $\pi$ from $Q(x,\cdot)$ to $Q(y,\cdot)$, since $Q(x,y)<Q(y,z)$, we must have $\pi(x,z)>0$. This means $W_\infty(Q(x,\cdot),Q(y,\cdot))=2$ and $K_{SEC}(x,y)=-1$. Thus Ollivier sectional curvature is negative.

\subsection{Perturbed $C_6$}

{\it The following example of a cycle of length $6$ has strictly positive entropic curvature but at least one edge with strictly negative Ollivier Ricci curvature (and therefore also strictly negative Ollivier sectional curvature) and at least one vertex with strictly negative Bakry-\'Emery curvature.}

\medskip

We start with a Markov chain $(X,Q_0,\pi_0)$ representing a lazy simple random walk on the cycle graph $C_6$, that is, $X=\mathbb{Z}_6$, $Q_0(i,i\pm 1)=q$ for some fixed $q<\frac{1}{4}$, and $Q(i,i)=1-2q>\frac{1}{2}$. Theorem \ref{thm:abellowbdinfdim} guarantees that 
$K_{ent}(X,Q_0,\pi_0) \ge \frac{q}{50\cdot 6^4}$. 

Let $\varepsilon>0$. Consider a perturbed Markov chain $(X,Q,\pi)$ with 
\[
Q(i,j):= \begin{cases}
    q+\varepsilon &,\text{if }(i,j)=(0,5) \\
    1-2q-\varepsilon &,\text{if }(i,j)=(0,0) \\
    Q_0(i,j) &,\text{otherwise}.
\end{cases}
\]
Our perturbation result in Theorem \ref{thm:perturb2} guarantees that, with small enough $\varepsilon$, the entropic curvature $K_{ent}(X,Q,\pi)$ stays strictly positive.

On the other hand, since both Ollivier Ricci curvature and Bakry-\'Emery curvature are locally defined curvatures, we can use the formulas for birth-death processes for their computation on large enough cycles. By Theorem \ref{thm:OL_line_formula}, we have
\[
K_{ORC}(0,1)=Q(0,1)+Q(1,0)-Q(1,2)-Q(0,5)=-\varepsilon.
\]
For Bakry-\'Emery curvature, by Theorem \ref{thm:BE_line_formula}, we have
\begin{align*}
    W_-(0)&=-Q(5,4)+3Q(5,0)+Q(0,5)-Q(0,1)=2q+\varepsilon, \\
    W_+(0)&=-Q(1,2)+3Q(1,0)+Q(0,1)-Q(0,5)=2q-\varepsilon,
\end{align*}
and $K_{BE}(0)=\max K$ such that $2q+\varepsilon-2K\ge 0$ and $(2q+\varepsilon-2K)(2q-\varepsilon-2K) \ge 4q(q+\varepsilon)$. It follows that $K_{BE}(0)<0$.

\subsection{Layer-rescaled prism $C_n \times K_2$ }

{\it The following example has negative entropic and Bakry-\'Emery curvature, but positive Ollivier Ricci curvature and non-negative Ollivier sectional curvature. The Bakry-\'Emery and entropic curvature are both negative because the gradient is distorted compared to the Lipschitz constant appearing in Ollivier curvatures.}

\medskip

Let $(X_0,Q_0,\pi_0)$ be a Markov chain. Let $r_1>r_2>0$ and $q>0$ with $r_1 + q \le 1$. We construct the following Markov chain $(X,Q,\pi)$ such that
$X = X_0 \times \{1,2\}$ and 
\begin{align*}
Q((x,1),(x,2)) = Q((x,2),(x,1)) &= q, \\
Q((x,k),(y,k)) &= r_k Q_0(x,y), \\
\pi(x,k) &= \pi_0(x),
\end{align*}
for all $x,y\in X_0$ and $k=1,2$. 

Let $\psi_0:X_0\to\mathbb{R}$ be an eigenfunction of the Laplacian on $(X_0,Q_0.\pi_0)$ to the eigenvalue $\lambda$, that is, $\Delta \psi_0 +\lambda \psi_0=0$. Define a function $\psi:X\to\mathbb{R}$ as
$$ \psi(x,k) := \psi_0(x) \qquad \forall x,k.$$
We set
$$ \mathcal{E}_0 := \langle \Gamma \psi_0, 1 \rangle_{\pi_0}= \sum_{x \in X_0} \Gamma \psi_0(x) \pi_0(x). $$
Let $\rho_1,\rho_2 > 0$ and define $\rho:X\to\mathbb{R}_+$ as $\rho(x,k) := \rho_k$ for all $x$ and $k$. Then we have
\begin{align*}
\Delta \psi(x,k) &= r_k\Delta \psi_0(x) = - \lambda r_k \psi(x,k),  \\
\Gamma_\rho \psi(x,k) &= r_k \Gamma \psi_0(x), \\ 
\Gamma_\rho(\psi,\Delta \psi)(x,k) &= - \lambda r_k^2 \Gamma \psi_0(x), \\
\langle \Delta \rho, \Gamma_\rho \psi \rangle_\pi &= -\sum_{x\in X_0}(\rho_2-\rho_1)(r_2-r_1)\Gamma \psi_0(x)q\pi_0(x)=q(\rho_2-\rho_1)(r_1-r_2)\mathcal{E}_0, \\
\langle \rho, \Gamma_\rho(\psi,\Delta \psi) \rangle_\pi &= -\lambda(\rho_1 r_1^2 + \rho_2 r_2^2) \mathcal{E}_0, \\
\mathcal{A}_\rho(\psi) &= \langle \rho,\Gamma_\rho \psi \rangle_\pi = \left( \rho_1 r_1 + \rho_2 r_2 \right) \mathcal{E}_0, \\
\mathcal{B}_\rho(\psi) &= \langle \rho, \Gamma_{2,\rho} \psi \rangle_\pi = \frac{1}{2} \langle \Delta \rho, \Gamma_\rho\psi\rangle_\pi - \langle \rho, \Gamma_\rho(\psi,\Delta \psi)\rangle_\pi \\
&= \left( \frac{1}{2} q(\rho_2-\rho_1)(r_1-r_2) + \lambda(\rho_1 r_1^2 + \rho_2 r_2^2) \right) \mathcal{E}_0,
\end{align*} so it follows that 
\[
K_{\theta}(X,Q,\pi) \le \frac{\mathcal{B}_\rho(\psi)}{\mathcal{A}_\rho(\psi)} = \frac{\frac{1}{2} q(\rho_2-\rho_1)(r_1-r_2) + \lambda(\rho_1 r_1^2 + \rho_2 r_2^2)}{\rho_1 r_1 + \rho_2 r_2}.
\]
Choosing $r_1= 2\epsilon > r_2 = \epsilon$ and $q=1/2$, we obtain
$$ K_\theta(X,Q,\pi) = \frac{\frac{1}{4}(\rho_2-\rho_1)+\epsilon\lambda\left(4 \rho_1+\rho_2 \right)}{2\rho_1+\rho_2}, 
$$
which becomes strictly negative for $\rho_1 > \rho_2$ and $\epsilon >0$ sufficiently small. This holds independently of the structure of the original Markov chain $(X_0,Q_0,\pi_0)$ and independently of the choice of the mean $\theta$. Thus we can conclude $K_{ent}(X,Q,\pi)<0$ and $K_{BE}(X,Q,\pi)<0$ simultaneously.

Next, we use formula \eqref{eq:OR_formula} for the computation of Ollivier Ricci curvature. 

We note that 
\begin{align*}
    \Delta f(x,1)&=r_1\Delta_0 f(x,1)+(f(x,2)-f(x,1))q, \\
    \Delta f(x,2)&=r_2\Delta_0 f(x,2)+(f(x,1)-f(x,2))q,
\end{align*}
where $\Delta_0$ is the standard Laplacian with respect to $(X_0,Q_0,\pi_0)$.

For $f\in Lip(1)$, we have $|\Delta_0 f(x,k)| \le 1$, so the formula \eqref{eq:OR_formula} yields
\[
K_{ORC}((x,1) , (x,2)) \ge 2q-(r_1+r_2)\sum_{y\in X_0}Q_0(x,y) \ge 2q-(r_1+r_2) >0.
\]
Moreover, consider $f$ to be the minimizer of \eqref{eq:OR_formula} for $K_{ORC}((x,1) , (y,1))$. Then
\begin{align*}
    K_{ORC}((x,1) , (y,1)) &= \inf_f r_1(\Delta_0 f(x,1)-\Delta_0 f(y,1))+ q(f(x,2)-f(y,2)+1) \\
    &\ge r_1 K_{ORC}(x,y).
\end{align*} Similarly, $K_{ORC}((x,2) , (y,2)) \ge r_2 K_{ORC}(x,y)$.

In particular, if the Markov chain $(X_0,Q_0,\pi_0)$ is chosen to be the simple random walk (with laziness $\frac{1}{2}$) on the cycle $C_n$ with $n\le 5$, it is well known that $K_{ORC}(X_0,Q_0,\pi_0)>0$. This implies Ollivier Ricci curvature on $(X,Q,\pi)$ is strictly positive.

Finally, we compute Ollivier sectional curvature. On any vertical edge $(x,1)\sim (x,2)$, we construct a transport plan $\pi$ from $Q((x,1),\cdot)$ to $Q((x,2),\cdot)$ by $\pi((v,1),(v,2))=\frac{1}{4}r_2$ and $\pi((v,1),(x,1)=\frac{1}{4}(r_1-r_2)$ for $v=x\pm1$ and the rest is transported between $(x,1)$ and $(x,2)$. This means $K_{SEC}((x,1),(x,2))=0$. On any horizontal edge $(x,k)\sim (x+1,k)$, for $k\in\{1,2\}$, we construct a transport plan $\pi$ from $Q((x,k),\cdot)$ to $Q((x+1,k),\cdot)$ simply by inducing the translation map $v\mapsto v+1$ fpr all $v$. This also gives $K_{SEC}((x,k),(x+1,k))=0$. This shows that the layer-rescaled prism $C_5 \times K_2$ has nonnegative Ollivier sectional curvature.

\subsection{Bernoulli-Laplace models}

In this subsection we summarize curvature considerations and the Modified Logarithmic Sobolov Inequality (MLSI) for Bernoulli-Laplace models from the literature. We start with the definition of Bernoulli-Laplace models.

\begin{definition} A \emph{Bernoulli-Laplace model} is  
a finite Markov chain $(X,Q,\pi)$ with the following properties: $X$ is the set of subsets of $[L]:=\{1,\dots,L\}$ of cardinality $N$ with $1 \le N \le L-1$.
The transition rates $Q$ (and the associated unique steady state $\pi$) are defined via a collection $\lambda_j \ge 0$, $j \in [L]$ of ``intensities'', as follows: For $x,y \in X$, $x \neq y$, we define
$$ Q(x,y) = \begin{cases} \lambda_j/L, & \text{if $x \setminus y =\{j\}$ for some $j \in [L]$,} \\ 0, & \text{otherwise,}  \end{cases} $$
and, for $x \in X$
$$ Q(x,x) = 1- \sum_{y \neq x} Q(x,y). $$
The intensities are assumed to be small enough such that
$$ \inf_{x \in X} Q(x,x) \ge \frac{1}{2}. $$
\end{definition}

Stochastically, a Bernoulli-Laplace model represents an interacting particle system with $N$ particles distributed on $L$ sites with the constraint that each site can contain at most one particle and where the particle on site $i$ jumps to an unoccupied sites with jump rate only depending on the $\lambda_i$.

Note that the underlying combinatorial graph of a Bernoulli-Laplace model is the Johnson Graph $J(L,N)$. Bakry-\'Emery curvature for the unweighted Johnson graph was given in \cite{KKRT-16} and \cite{CLP-19}.

To state the Modified Logarithmic Sobolev Inequality, we need to introduce the following functionals for $f,g \in \mathbb{R}^X$ and $\rho \in \mathcal{P}_*(X)$ (that is, $\rho > 0$ with $\sum_{x \in X} \rho(x) \pi(x) = 1$):
\begin{align*}
\mathcal{E}(f,g) &:= \langle \nabla f, \nabla g \rangle_\pi = \frac{1}{2}   \sum_{x,y \in X} Q(x,y) \pi(x) (f(y)-f(x))(g(y)-g(x)), \\
{\mathrm{Ent}}(\rho) &:= \langle \rho, \log\rho\rangle_\pi = \sum_{x \in X} \rho(x)\log \rho(x) \pi(x).
\end{align*}
We say that $(X,Q,\pi)$ satisfies (MLSI) with optimal constant $\alpha_{MLSI} > 0$ where
$$ \alpha_{MLSI} := \inf_{\rho \in \mathcal{P}_*(X)} \frac{\mathcal{E}(\rho,\log \rho)}{\mathrm{Ent}(\rho)}. $$
The following results hold for Bernoulli-Laplace models:
\begin{enumerate}[(i)]
    \item It was shown in \cite[Theorem 5.2]{CDPP-09} that (MLSI) holds with 
    $$ \alpha_{MLSI} \ge \min \lambda - (\max \lambda - \min \lambda).
    $$
    \item Assuming $\max \lambda \le 3 \min \lambda$, it was shown in \cite[Thm. 4.7]{FM-16} that the \emph{entropic curvature} is lower bounded by
    $$ K = \frac{\min \lambda}{2} - \frac{7}{8} \left( \max \lambda - \min \lambda \right). $$
    \item Employing the general relation between entropic curvature and the MLSI constant for any finite Markov chain in \cite[Theorem~1.5]{EM-12}, (ii) implies 
    $$\alpha_{MLSI} \ge 2 K = \min \lambda - \frac{7}{4}(\max \lambda - \min \lambda),$$ 
    which is slightly weaker than (i).
    \item It was discussed at the end of \cite{CMS-25} that employing Ollivier Ricci curvature and sectional curvature gives better curvature (see Theorem \ref{thm:CMS} below) and MLSI constant estimates, which do not depend on $\max \lambda$. Namely, non-negative sectional curvature and \eqref{eq:KORCBLM} below imply the improved inequality
    $$ \alpha_{MLSI} \ge \inf_{x \sim y} K_{ORC}(x,y) \ge  \frac{1}{L} \left( (L-N-1) \lambda_1 + \sum_{i=1}^{N+1} \lambda_i \right) \ge \min \lambda $$
    by \cite[Theorem 4.4]{munch2023ollivier}. Here we assume the following ordering of the intensities
    $$ 0 \le \lambda_1 \le \lambda_2 \le \ldots \le \lambda_L. $$
    Note that even in the case $\min \lambda = 0$, this lower bound is still meaningful, for example, in the case $L=2N$ and $\lambda_i=0$ for $i \le L/4$ and $\lambda_i \ge 1$ for $i > L/4$, we have
    $$ \alpha_{MLSI} \ge \frac{1}{4}. $$
\end{enumerate}    

\begin{theorem} \label{thm:CMS} Let $(X,Q,\pi)$ be a Laplace-Bernoulli model with $N$ particles and $L$ sites and corresponding intensities 
$$ 0 \le \lambda_1 \le \lambda_2 \le \ldots \le \lambda_L. $$
Then $(X,Q,\pi)$ has non-negative Ollivier sectional curvature $\inf_{x \sim y} K_{SEC}(x,y) \ge 0$ and we have 
\begin{equation} \label{eq:KORCBLM}
\inf_{x \sim y} K_{ORC}(x,y) \ge \frac{1}{L}\left( (L-N-1) \lambda_1 + \sum_{i=1}^{N+1} \lambda_i \right). 
\end{equation}
\end{theorem}

\begin{proof}
    We start with the computation of $K_{ORC}(x,y)$ where the labels on vertices $x$ and $y$ are $x=(\underbrace{1,1,...,1}_{N},\underbrace{0,0,...,0}_{L-N})$ and $y=(0,\underbrace{1,1,...,1}_{N},\underbrace{0,0,...,0}_{L-N-1})$. There are two types of vertices $z$ which are common neighbors of $x$ and $y$ (that is, $z\in S_1(x)\cap S_1(y)$):
\begin{itemize}
    \item $z=(0,\underbrace{1,1,...,1}_{N-1 },0,0,0,...,0,1,0...,0)$ where the last $1$ is at position $N+2\le i \le L$. In this case, we have $Q(x,z)=\lambda_1$ and $Q(y,z)=\lambda_{N+1}$.
    \item $z=(\underbrace{1,1,...,1,0,1,...,1}_{N+1},0,0,...,0)$ where the first $0$ is at position $2\le j \le N$. In this case, we have $Q(x,z)=Q(y,z)=\lambda_j$.
\end{itemize}
For vertices $z$ which are not common neighbors of $x$ and $y$, there is a perfect matching 
$\Phi$ from  
$S_1(x)\setminus S_1(y)$ to $S_1(y)\setminus S_1(x)$, namely 
$$\Phi(\tilde{x}):=\tilde{x}+y-x,$$
by applying the tuple-addition. Moreover, we have
$$ Q(x,\tilde x) = Q(y, \Phi(\tilde x)). $$
Hence, there exists a transport plan $\gamma$ from $Q(x,\cdot)$ to $Q(y,\cdot)$ with maximum transport distance $1$, which implies non-negative sectional curvature $K_{SEC}(x,y) \ge 0$. This transport plan is optimal and given by
$$
\gamma(\tilde x,\tilde y) = \begin{cases} 
Q(x,y), &\text{if $\tilde x =\tilde y =y$,} \\
Q(y,x), &\text{if $\tilde x =\tilde y =x$,} \\
Q(x,\tilde x) \wedge Q(y,\tilde y), &\text{if $\tilde x =\tilde y \in S_1(x) \cap S_1(y)$,} \\
Q(x,\tilde x), &\text{if $\tilde x \in S_1(x)\setminus S_1(y)$ and $\tilde y = \Phi(\tilde x)$,} \\
Q(x,\tilde x)-(Q(x,\tilde x) \wedge Q(y,\tilde y)) & \text{if $\tilde x \in S_1(x) \cap S_1(y)$ and $\tilde y =y$,} \\
Q(y,\tilde y)-(Q(x,\tilde x) \wedge Q(y,\tilde y)) & \text{if $\tilde x =x $ and $\tilde y \in S_1(x) \cap S_1(y)$,} \\
c, & \text{if $\tilde x=x$ and $\tilde y =y$,} \\
0, & \text{otherwise.}
\end{cases}
$$
Here, $c \ge 0$ is chosen such that $\sum_{\tilde x,\tilde y \in X} \gamma(\tilde x, \tilde y) = 1$.
The Ollivier-Ricci curvature is therefore the total mass transported by distance $0$, that is,
\begin{align*}
    K_{ORC}(x,y) 
    &= Q(x,y)+Q(y,x)+\sum_{z\in S_1(x)\cap S_1(y)} Q(x,z) \wedge Q(y,z)\\
    &=\frac{1}{L}\left(\lambda_1+\lambda_{N+1}+(L-N-1)\min\lambda + \sum_{j=2}^N \lambda_j\right)\\
    &=\frac{1}{L}\left( (L-N-1) \lambda_1 + \sum_{i=1}^{N+1} \lambda_i \right).
\end{align*}
Analogously, we can compute Ollivier-Ricci curvature of any edge $u\sim v$ as
\begin{align*}
    K_{ORC}(u,v) &=\frac{1}{L}\left( (L-N-1) \lambda_{\sigma(1)} + \sum_{i=1}^{N+1} \lambda_{\sigma(i)} \right) \\
    &\ge \frac{1}{L}\left( (L-N-1) \lambda_{1} + \sum_{i=1}^{N+1} \lambda_i \right),
\end{align*} 
where $\sigma \in \mathcal{S}_L$ is a suitable permutation.
\end{proof}



\section{Some open questions}
\label{sec:questions}

We like to finish our paper with a few question for further investigation:
\begin{itemize}
\item It might be interesting to investigate entropic curvature of the simple random walk on the dodecahedron. We expect positive entropic curvature. It this would be the case, we would have an example of a Markov chain with positive entropic curvature and strictly negative Bakry-\'Emery curvature everywhere.
\item In our considerations concerning cycles, an optimal $\rho$ seems to be a Gaussian-type distribution. This raises the questions whether, for general Markov chains, an optimal $\rho$ satisfies $\rho = P_t \rho_0$ with some other positive function $\rho_0$ and $P_t$ the heat semigroup, and how large can $t$ be chosen.  
\item In the case of Bakry-\'Emery curvature, we have $\lim_{n \to 0} K_n(x) = -\infty$. Does a similar fact hold for the non-local entropic curvature as dimension tends to $-\infty$?
\item Which combinatorial graphs with $n$ vertices, equipped with the simple random walk, have entropic curvature (with infinite dimension) closest to $0$? Is there any better candidate than the cycle (which has positive entropic curvature)?
\item Do all Ricci flat or (S)-Ricci flat graphs, as defined in Definition \ref{def:ricciflat}, have non-negative entropic curvature, as it is the case for Ollivier Ricci and Bakry-\'Emery curvature?
\item Given a Markov chain and a small interval, what is the computational complexity for the decision problem whether its entropic curvature is within this interval? Is this an NP-hard problem?
\item For the entropic curvature estimate of the Bernoulli-Laplace model, the article \cite{FM-16} required the restriction $\frac{\max \lambda}{\min \lambda} \le 3$. It would be interesting to investigate both entropic curvature and Bakry-\'Emery curvature in the case when $\frac{\max \lambda}{\min \lambda}$ is large.
\end{itemize}

{\bf{Acknowledgements:}} We thank Matthias Erbar for stimulating discussions. Shiping Liu is supported by the National Key R and D Program of China 2020YFA0713100 and the National Natural Science Foundation of China No. 12031017. Supanat Kamtue is supported
by Shuimu Scholar Program of Tsinghua University No. 2022660219.

\appendix

\printbibliography

\end{document}